\documentclass[12pt]{article}
\usepackage[letterpaper,margin=1in]{geometry}
\usepackage{amsmath,amssymb,amsthm,mathtools}
\usepackage{enumitem}
\usepackage{tikz}
\usepackage{graphicx}
\usepackage{subcaption}
\usepackage[colorlinks=true,linkcolor=blue,citecolor=blue,urlcolor=blue]{hyperref}

\newtheorem{theorem}{Theorem}
\newtheorem{lemma}[theorem]{Lemma}
\newtheorem{proposition}[theorem]{Proposition}
\newtheorem{corollary}[theorem]{Corollary}
\newtheorem{remark}[theorem]{Remark}
\newtheorem{example}[theorem]{Example}

\newcommand{\Aut}{\mathrm{Aut}}

\title{\large The existence and uniqueness of magic-faced~hypercubes, and applications to Khajuraho most-perefect magic~squares, cubes, and hypercubes}
\author{\normalsize Manjul Bhargava}
\date{\small \today}

\begin{document}
\maketitle

\vspace{-.25in}
\begin{abstract}
A \emph{magic-lined hypercube} (or, simply, {\it magic hypercube}) of order $k$ and dimension~$n$ is an arrangement of the numbers $1,\dots,k^n$ in a $k\times\cdots\times k$ ($n$-fold) grid such that every line of $k$ numbers parallel to a coordinate axis has the same magic sum. While such hypercubes exist for every order $k\ge 3$ and every dimension $n$, no magic-lined hypercube of order~$2$ exists in any dimension $n\ge 2$. 

For hypercubes of order~$2$, we thus relax the magic  condition from lines to two-dimensional faces or planes. We call an arrangement of the numbers $1,\dots,2^n$ in a $2\times\cdots\times2$ ($n$-fold) grid \emph{magic-faced} if every $2\times2$ face has the same magic sum. We prove that a magic-faced hypercube of order~$2$ exists in every dimension $n\ge0$, and that it is unique up to a certain natural set of transformations of size $2^n(n+1)!$ when $n$ is even and $2^nn\cdot n!$ when $n$ is odd. 

As~an~application, we recover and generalize the classical $4\times4$ Khajuraho magic square, \,\!give a new, conceptual answer to a question of Coxeter on the group acting on the $384$ ``most-perfect" $4\times4$ magic squares, \,\!and extend the picture to higher dimensions. In particular, we prove that, in dimension~$n$, these most-perfect hypercubes form a single orbit under a certain natural action of the Weyl group $W(B_{2n})$. 
\end{abstract}

\section{Introduction}
\label{sec:intro}

A {\it magic-lined hypercube} (or, simply, a {\it magic hypercube}) of order~$k$ and dimension~$n$ is an arrangement of the numbers $1,\dots,k^n$ in a $k\times\cdots\times k$ ($n$-fold) matrix such that any $k$ numbers lying on a line parallel to one of the $n$ axes have the same sum, called the \emph{magic sum}. (Some authors also require the main diagonals to sum to the magic sum in the definition of a magic hypercube.)

It is known that magic-lined hypercubes of every order $k\ge3$ and every dimension $n$ exist. However, it is also known that magic-lined hypercubes of order $2$ do \emph{not} exist in any dimension $n\ge2$. For example, there is no $2\times2$ magic square, and no $2\times2\times2$ magic-lined cube. This suggests that the condition of being magic-lined is too strong a condition for hypercubes of order $2$.


It is thus natural to relax the magic sum condition for order-2 hypercubes from lines to planes.  We call an arrangement of the numbers $1,\ldots,2^n$ in a $2\times2\times \cdots \times 2$ ($n$-fold) matrix a {\it  magic-faced hypercube} if every $2\times 2$ square face has the same sum.  

The purpose of this article is to demonstrate the existence and uniqueness, up to a natural set of symmetries, of magic-faced hypercubes in every dimension $n\ge0$. The uniqueness  points to the naturality of the magic-faced condition, and leads to a new type of combinatorial design for hypercubes of order $2$.

\begin{theorem}
\label{thm:main}
There exists a magic-faced hypercube in every dimension $n\ge0$. Moreover: \vspace{-.045in}
\begin{enumerate}
\item[\rm(a)] If $n$ is even, then a magic-faced hypercube of dimension $n$ is unique up to the action of a permutation group $G_n \subset S_{2^n}$, where $G_n \cong C_2^n \rtimes S_{n+1}$. 
That is, the magic-faced hypercubes of dimension $n$ form a single orbit of $G_n.$ \vspace{-.045in}
\item[\rm(b)] If $n$ is odd, then any magic-faced hypercube of dimension $n$ has a distinguished coordinate direction $j\in\{1,\ldots,n\}$ such that every pair of numbers on a line in that direction sums to~$2^n+1$. A magic-faced hypercube of dimension $n$ and direction $j$ is unique up to the action of a permutation group $G_{n,j}  \subset S_{2^n}$, where $G_{n,j} \cong  \bigl(C_2^{n-1}\rtimes S_n\bigr)\times C_2$.  
That is, the magic-faced hypercubes of dimension $n$ and direction $j$ form a single orbit of~$G_{n,j}.$ 
\end{enumerate}\vspace{-.045in}
Hence there are $2^n(n+1)!$ magic-faced hypercubes of dimension $n$ when $n$ is even and $n\cdot 2^n n!$ when $n$ is odd. 
\end{theorem}

Our proof leads naturally to an explicit construction of the essentially unique such magic-faced hypercube in each dimension $n$; the cases $n\le4$ are illustrated in Figure~\ref{mfh}.

\vspace{.03in}

\begin{figure}[htbp]
\label{mfh}
\centering
\begin{subfigure}[b]{0.10\textwidth}
\centering
\begin{tikzpicture}[every path/.style={draw=red}, every node/.style={text=blue}, scale=0.8, baseline=(current bounding box.center)]
\node[fill=white,inner sep=2pt] at (0,1) {$1$};
\end{tikzpicture}
\caption*{dim 0}
\end{subfigure}
\hspace{0.17\textwidth}
\begin{subfigure}[b]{0.16\textwidth}
\centering
\begin{tikzpicture}[every path/.style={draw=red}, every node/.style={text=blue}, scale=0.6, baseline=(current bounding box.center)]
\draw (0,1) -- (2,1);
\node[fill=white,inner sep=2pt] at (0,1) {$1$};
\node[fill=white,inner sep=2pt] at (2,1) {$2$};
\end{tikzpicture}
\caption*{dim 1}
\end{subfigure}
\hspace{0.17\textwidth}
\begin{subfigure}[b]{0.20\textwidth}
\centering
\begin{tikzpicture}[every path/.style={draw=red}, every node/.style={text=blue}, xscale=.795, yscale=.745, baseline=(current bounding box.center)]
\draw (0,0) -- (2,0) -- (2,2) -- (0,2) -- cycle;
\node[fill=white,inner sep=2pt] at (0,0) {$1$};
\node[fill=white,inner sep=2pt] at (2,0) {$3$};
\node[fill=white,inner sep=2pt] at (2,2) {$4$};
\node[fill=white,inner sep=2pt] at (0,2) {$2$};
\end{tikzpicture}
\caption*{dim 2}
\end{subfigure}

\vspace{2.05em}

\begin{subfigure}[b]{0.34\textwidth}
\centering
\raisebox{1.1cm}{%
\begin{tikzpicture}[every path/.style={draw=red}, every node/.style={text=blue}, xscale=0.5724, yscale=0.54]
\draw (0,0) -- (3,0) -- (3,3) -- (0,3) -- cycle;
\draw (1.3,1.3) -- (2.85,1.3);
\draw (3.15,1.3) -- (4.3,1.3);
\draw (4.3,1.3) -- (4.3,4.3);
\draw (4.3,4.3) -- (1.3,4.3);
\draw (1.3,4.3) -- (1.3,3.15);
\draw (1.3,2.85) -- (1.3,1.3);
\draw (0,0) -- (1.3,1.3);
\draw (3,0) -- (4.3,1.3);
\draw (3,3) -- (4.3,4.3);
\draw (0,3) -- (1.3,4.3);
\node[fill=white,inner sep=2pt] at (0,0) {$1$};
\node[fill=white,inner sep=2pt] at (3,0) {$8$};
\node[fill=white,inner sep=2pt] at (3,3) {$3$};
\node[fill=white,inner sep=2pt] at (0,3) {$6$};
\node[fill=white,inner sep=2pt] at (1.3,1.3) {$4$};
\node[fill=white,inner sep=2pt] at (4.3,1.3) {$5$};
\node[fill=white,inner sep=2pt] at (4.3,4.3) {$2$};
\node[fill=white,inner sep=2pt] at (1.3,4.3) {$7$};
\node[text=black] at (2.15,-1.0) {dim 3};
\end{tikzpicture}}
\end{subfigure}
\hspace{2em}
\begin{subfigure}[b]{0.55\textwidth}
\centering
\begin{tikzpicture}[every path/.style={draw=red}, every node/.style={text=blue}, xscale=0.4225, yscale=.4125, baseline=(current bounding box.center)]
\draw (0,0) -- (10,0);
\draw (0,0) -- (0,10);
\draw (10,0) -- (10,10);
\draw (0,10) -- (10,10);
\draw (2.5,1.5) -- (8.52,1.5);
\draw (9.08,1.5) -- (9.72,1.5);
\draw (10.28,1.5) -- (12.5,1.5);
\draw (12.5,1.5) -- (12.5,11.5);
\draw (12.5,11.5) -- (2.5,11.5);
\draw (2.5,11.5) -- (2.5,10.28);
\draw (2.5,9.72) -- (2.5,8.306);
\draw (2.5,7.746) -- (2.5,2.912);
\draw (2.5,2.352) -- (2.5,1.5);
\draw (0,0) -- (2.5,1.5);
\draw (10,0) -- (12.5,1.5);
\draw (10,10) -- (12.5,11.5);
\draw (0,10) -- (2.5,11.5);
\draw (3.8,4) -- (6.8,4) -- (6.8,7) -- (3.8,7) -- cycle;
\draw (5.675,5.125) -- (6.52,5.125);
\draw (7.08,5.125) -- (8.675,5.125);
\draw (8.675,5.125) -- (8.675,8.125);
\draw (5.675,8.125) -- (7.72,8.125);
\draw (8.28,8.125) -- (8.675,8.125);
\draw (5.675,8.125) -- (5.675,7.28);
\draw (5.675,6.72) -- (5.675,5.125);
\draw (3.8,4) -- (5.675,5.125);
\draw (6.8,4) -- (8.675,5.125);
\draw (6.8,7) -- (8.675,8.125);
\draw (3.8,7) -- (5.675,8.125);
\draw (0,0) -- (3.8,4);
\draw (10,0) -- (6.8,4);
\draw (10,10) -- (6.8,7);
\draw (0,10) -- (3.8,7);
\draw (2.5,1.5) -- (4.505,3.789);
\draw (4.874,4.211) -- (5.675,5.125);
\draw (12.5,1.5) -- (10.2,3.677);
\draw (9.797,4.062) -- (8.675,5.125);
\draw (12.5,11.5) -- (10.21,9.479);
\draw (9.79,9.109) -- (8.675,8.125);
\draw (2.5,11.5) -- (3.719,10.2);
\draw (4.103,9.796) -- (5.675,8.125);
\node[fill=white,inner sep=2pt] at (0,0) {$1$};
\node[fill=white,inner sep=2pt] at (10,0) {$15$};
\node[fill=white,inner sep=2pt] at (10,10) {$4$};
\node[fill=white,inner sep=2pt] at (0,10) {$14$};
\node[fill=white,inner sep=2pt] at (2.5,1.5) {$12$};
\node[fill=white,inner sep=2pt] at (12.5,1.5) {$6$};
\node[fill=white,inner sep=2pt] at (12.5,11.5) {$9$};
\node[fill=white,inner sep=2pt] at (2.5,11.5) {$7$};
\node[fill=white,inner sep=2pt] at (3.8,4) {$8$};
\node[fill=white,inner sep=2pt] at (6.8,4) {$10$};
\node[fill=white,inner sep=2pt] at (6.8,7) {$5$};
\node[fill=white,inner sep=2pt] at (3.8,7) {$11$};
\node[fill=white,inner sep=2pt] at (5.675,5.125) {$13$};
\node[fill=white,inner sep=2pt] at (8.675,5.125) {$3$};
\node[fill=white,inner sep=2pt] at (8.675,8.125) {$16$};
\node[fill=white,inner sep=2pt] at (5.675,8.125) {$2$};
\end{tikzpicture}
\caption*{dim 4}
\end{subfigure}
\caption{The essentially-unique magic-faced hypercube of order $2$ in dimensions $n=0,1,2,3,4$. 
Every $4$-cycle (i.e.,  every $2$-dimensional face) of each graph sums to the same magic constant, namely, $2(2^n+1)$. 
Thus each $2\times 2$ face of the magic-faced cube has sum~$18$, while each $2\times 2$ face of the magic-faced tesseract has sum~$34$. 
}
\label{mfh}
\end{figure}

In fact, a stronger structural statement regarding such magic-faced hypercubes holds, which in turn explains the shape of the symmetry group/set $G_n$ appearing in Theorem~\ref{thm:main}.  To state it, we observe that a magic-faced hypercube may be viewed as a labeling of the vertices of the hypercube graph $Q_n$ of dimension $n$ with the integers $1,\ldots,2^n$ such that the labels on every $4$-cycle have the same magic sum.  

The {\it folded hypercube graph} $FQ_n$ of dimension $n$ is the unique augmentation of the hypercube graph $Q_n$ where every pair of antipodal points of $Q_n$ has also been joined by an~edge.  

Similarly, the {\it hierarchical folded hypercube graph}  $H\!FQ_n$ of dimension $n$ is an augmentation of $Q_n$ where additional edges are added connecting pairs of antipodal points of each of two disjoint copies of $Q_{n-1}$ contained in $Q_n$. That is  $H\!FQ_n:=K_2\times FQ_{n-1}$. There are $n$ possible such augmentations $H_jFQ_n$, all isomorphic as graphs, corresponding to the $n$-possible coordinate directions $j$ along which $Q_n$ is broken into two copies of $Q_{n-1}$. 

A {\it magic-faced folded hypercube} (respectively, {\it magic-faced hierarchical folded hypercube}) of dimension $n$ is a labeling of $FQ_n$ (respectively, $H\!FQ_{n}$) with the integers $1,\ldots,2^n$ such that the labels of every 4-cycle have the same magic sum.  Our stronger structural statement about magic-faced hypercubes is then as follows: 

\begin{theorem}
\label{thm:folded}
Every magic-faced hypercube is automatically a magic-faced folded hypercube if $n$ is even, and a magic-faced hierarchical folded hypercube if $n$ is odd.   
\end{theorem}

For example, in the dim 2 magic-faced square in Figure~\ref{mfh}, we may add the edges connecting (1,4) and (2,3) to obtain the complete graph $K_4=FQ_2$ on 4 vertices.  The sum of the labels on any 4-cycle of this complete graph indeed have the same sum of 10, since these labels  always consist of the numbers 1 through 4.  Similarly, in the dim 3 magic-faced cube in Figure~\ref{mfh}, we add edges connecting  (1,7), (4,6), (2,8), and (3,5), to obtain the graph~$H\!FQ_3$. The labels on every 4-cycle of this augmented graph also have the same sum of 18, e.g., $1+2+8+7=18$.  Finally, in the dim~4 magic-faced tesseract in Figure~\ref{mfh}, we connect all antipodal pairs (1,16), (14,3), (7,10), etc.\ to obtain the graph $FQ_4$; the labels on every 4-cycle of this folded tesseract also then have the same sum of 34, e.g.,  $14+3+10+7=34$. 

In light of Theorem~\ref{thm:folded}, we may rephrase Theorem~\ref{thm:main} more elegantly as follows. 

\begin{theorem}
\label{thm:folded2}
\hfill
\begin{itemize}
\item[{\em (a)}] If $n$ is even $($resp.\ odd$)$, there exists a magic-faced folded hypercube $($resp.\ hierarchical folded hypercube$)$ $A_n$ of dimension $n$, which is unique up to the action of $\Aut(FQ_n)$ $($resp., $\Aut(H\!FQ_n))$, i.e., up to the automorphisms of the underlying graph. 

\item[{\em (b)}]
Magic-faced hypercubes of dimension $n$ are in one-to-one correspondence with graph embeddings 
\[\begin{cases}
\phi:Q_n\hookrightarrow FQ_n & \text{if } n \text{ is even},\\[2pt]
\phi:Q_n\hookrightarrow H\!FQ_n & \text{if } n \text{ is odd}.
\end{cases}
\]
Given such an embedding $\phi$, the label of a vertex $v\in Q_n$ is given by $A_n(\phi(v))$. 
\end{itemize}
\end{theorem}
\noindent
Theorem~\ref{thm:folded2} states that, up to automorphisms of the underlying graph, there is a unique magic-faced folded hypercube $A_n$ of dimension $n$ when $n$ is even, and a unique magic-faced hierarchical folded hypercube $A_n$ when $n$ is odd.  Moreover, every magic-faced hypercube of dimension $n$ is obtained by taking a hypercube graph $Q_n$ inside $A_n$ (with the attached~labels). 

\pagebreak 

Because all embeddings $Q_n\hookrightarrow FQ_n$ are related to one another via postcomposition with an element of $\Aut(FQ_n)\cong C_2^n \rtimes S_{n+1}$, this recovers $G_n$ as given in Theorem~\ref{thm:main} when $n$ is~even.  Similarly, the action of $\Aut(H\!FQ_n)\cong \bigl(C_2^{n-1}\rtimes S_n\bigr)\times C_2$ on the set of embeddings $Q_n\hookrightarrow H\!FQ_n$ via postcomposition breaks up into $n$ orbits (corresponding to the direction~$j$ along which we break up $Q_n$ into two $Q_{n-1}$'s); this gives the description of $G_n$ and $G_{n,j}$ as~in Theorem~\ref{thm:main} when $n$ is odd.  Since $\Aut(FQ_n)=2^{n}(n+1)!$ and $\Aut(HFQ_n)=2^{n}n!,$ 
this yields the cardinalities for the number of magic hypercubes of dimension $n$ for each $n$ as stated in~Theorem~\ref{thm:main}. 

\begin{remark}
{\em The proof of Theorem~\ref{thm:folded2} shows that there are no magic-faced folded hypercubes when $n>1$ is odd, and no magic-faced hierarchical folded hypercubes when $n>2$ is~even. Thus ``folded" and ``hierarchical folded" form natural conditions for $n$ even and $n$ odd, respectively, but not vice versa!}
\end{remark}

\vspace{-.115in}
\begin{remark}
{\em There are two alternative descriptions of the folded hypercube graph $FQ_n$ that make evident its $S_{n+1}$-symmetry.
The first is as the $(n+1)$-cube graph $Q_{n+1}$---with coordinate directions labelled $0,1,\ldots,n$---where pairs of vertices symmetrically opposite the center (i.e.,~antipodal pairs) are identified. 
The second is more algebraic: the set of vertices is taken to be $\mathbb F_2^{n+1}/\langle e_0+\cdots+e_n\rangle$, where $e_0,\ldots,e_n$ is the natural coordinate basis of $\mathbb F_2^{n+1}$; two vertices $v$ and $w$ are joined by an edge if $v-w\in\{[e_0],\ldots,[e_n]\}$. In both cases, $S_{n+1}$ permutes the coordinate directions $0,1,\ldots,n$, making the $S_{n+1}$-symmetry of $FQ_n$ evident.}
\end{remark}

\subsection*{\normalsize Associative Magic-Faced Hypercubes}

Common in the literature on magic hypercubes is the notion of ``associative". An {\it associative} magic hypercube is one in which any two numbers that are symmetrically opposite the center of the hypercube have the same~sum.  Classifying associative magic-faced hypercubes will turn out to be important in the~applications.

In an associative magic-faced hypercube of dimension~$n$, any antipodal pair of numbers must have sum $2^n+1$, i.e., twice the average of the labels~$1,\ldots,2^n$. Theorem~\ref{thm:main}  states that, if $n$ is odd, then a magic-faced hypercube of dimension $n$ has a coordinate direction $j\in\{1,\ldots,n\}$ such that every pair of numbers on a line in that direction sums to~$2^n+1$; hence it cannot be associative.  

In the case of even $n$, we analogously prove that a magic-faced hypercube of dimension~$n$ has a coordinate direction $j\in\{0,\ldots,n\}$ such that every pair of numbers on a line in that direction sums to~$2^n+1$ (where $j=0$ corresponds to the case of antipodal pairs). The group $S_{n+1}\subset G_n$ acts naturally on these $n+1$ possibilities.  Thus we obtain the following theorem: 
\begin{theorem}\label{thm:associative}
If $n$ is even, then there exists an associative magic-faced  hypercube of dimension~$n$ which is unique up to the symmetry group $\Aut(Q_n)\cong C_2^n \rtimes S_{n}$ of the hypercube. If~$n$~is~odd, then no associative magic-faced hypercube of dimension $n$ exists.
\end{theorem}
\noindent
In Figure~\ref{mfh}, in the cases of even dimension, the representative magic-faced hypercubes illustrated were indeed chosen to be associative; in the dim 2 case, pairs of antipodal entries all sum to 5, while in the dim 4 case, they all sum to 17. 
 
 Theorem~\ref{thm:associative} has important consequences for Khajuraho ``most-perfect" magic squares, cubes, and hypercubes, which we discuss in the next subsection.

\subsection*{\normalsize The Khajuraho Magic Square and a Question of Coxeter}
\label{subsec:khajuraho}

One of the original motivations for this work was to understand the existence of, and the remarkable and surprising properties enjoyed by, the first $4\times4$ magic square ever recorded in history: the \emph{Khajuraho magic square}, inscribed on the wall of the Parshvanath Jain temple at Khajuraho, India (10th century CE) \cite{DuttaSingh2019,BallCoxeter}---and to understand a question of Professor H.\,S.\,M.\ Coxeter regarding the existence and structure of such squares \cite{BallCoxeter}.

\begin{figure}[htbp]
\centering
\begin{subfigure}[b]{0.36\textwidth}
\centering
\includegraphics[width=3.6cm]{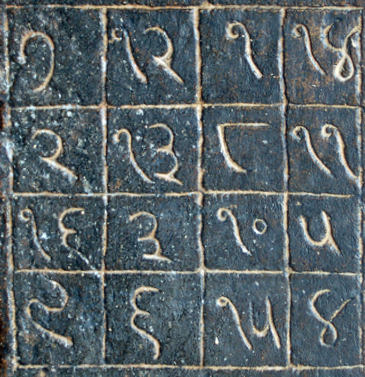}
\end{subfigure}
\hspace{0.01\textwidth}
\begin{subfigure}[b]{0.36\textwidth}
\centering
\begin{tikzpicture}[scale=0.7225]
\def\cs{1.25}
\draw[line width=0.9pt,xstep=\cs,ystep=\cs] (0,0) grid (4*\cs,4*\cs);
\node at (0.5*\cs,3.5*\cs) {\Large 7};
\node at (1.5*\cs,3.5*\cs) {\Large 12};
\node at (2.5*\cs,3.5*\cs) {\Large 1};
\node at (3.5*\cs,3.5*\cs) {\Large 14};
\node at (0.5*\cs,2.5*\cs) {\Large 2};
\node at (1.5*\cs,2.5*\cs) {\Large 13};
\node at (2.5*\cs,2.5*\cs) {\Large 8};
\node at (3.5*\cs,2.5*\cs) {\Large 11};
\node at (0.5*\cs,1.5*\cs) {\Large 16};
\node at (1.5*\cs,1.5*\cs) {\Large 3};
\node at (2.5*\cs,1.5*\cs) {\Large 10};
\node at (3.5*\cs,1.5*\cs) {\Large 5};
\node at (0.5*\cs,0.5*\cs) {\Large 9};
\node at (1.5*\cs,0.5*\cs) {\Large 6};
\node at (2.5*\cs,0.5*\cs) {\Large 15};
\node at (3.5*\cs,0.5*\cs) {\Large 4};
\end{tikzpicture}
\end{subfigure}
\caption{The Khajuraho Magic Square: the original inscription of the magic square (left), constructed over 1000 years ago, alongside the same square transcribed into the modern international form of Indian numerals (right).}
\label{fig:khajuraho}
\end{figure}

\vspace{-.05in}
In this magic square, not only does every row, column, and diagonal sum to the same magic constant $34$, but so does every $2\times2$ subsquare; in fact, every $2\times2$ subsquare does so even when allowing ``wrapping around,'' i.e., viewing the $4\times4$ square as a $4\times4$ toroidal grid $C_4\times C_4$.

How many such magic squares are there, i.e., for which every $2\times2$ subsquare (equivalently, every $2\times2$ subsquare of the square viewed as a torus) also sums to the magic constant $34$?

We approach this problem using Theorems~\ref{thm:main}--\ref{thm:folded2} by noting that there is an exceptional graph isomorphism between the toroidal $4\times4$ grid graph $C_4\times C_4$ and the tesseract graph~$Q_4$. 
Indeed, the cyclic graph $C_4$ on~$4$~vertices is isomorphic, as a graph, to $K_2\times K_2$; hence the toroidal $4\times4$ grid graph $C_4\times C_4$ is isomorphic to $K_2\times K_2\times K_2\times K_2$, i.e., to the tesseract graph~$Q_4$. Under this graph isomorphism, the eight orthogonal lines and sixteen $2\times2$ subsquares  on the toroidal $4\times 4$ graph correspond exactly to the twenty-four $2\times2$ faces ($4$-cycles) of the tesseract graph.  Hence a $4\times4$ toroidal magic square---all of whose rows, columns, and $2\times2$ subsquares sum to the same magic constant---is equivalent to a magic-faced tesseract!

\begin{figure}[htbp]
\centering
\begin{tikzpicture}[every path/.style={draw=red}, every node/.style={text=blue}, scale=0.6154]
\useasboundingbox[draw=none] (0,-1.7) rectangle (19,7.8);
\def\gs{2.6}
\foreach \r in {0,1,2,3}{
  \pgfmathsetmacro{\y}{(3-\r)*\gs}
  \draw (0,\y) -- (1*\gs,\y);
  \draw (1*\gs,\y) -- (2*\gs,\y);
  \draw (2*\gs,\y) -- (3*\gs,\y);
}
\foreach \c in {0,1,2,3}{
  \pgfmathsetmacro{\x}{\c*\gs}
  \draw (\x,0) -- (\x,1*\gs);
  \draw (\x,1*\gs) -- (\x,2*\gs);
  \draw (\x,2*\gs) -- (\x,3*\gs);
}
\foreach \r in {0,1,2,3}{
  \pgfmathsetmacro{\y}{(3-\r)*\gs}
  \draw (3*\gs,\y) .. controls (3*\gs+1.6,\y+1.0) and (-1.6,\y+1.0) .. (0,\y);
}
\draw (0,0) .. controls (-1.0,-1.6) and (-1.0,3*\gs+1.6) .. (0,3*\gs);
\draw (1*\gs,0) .. controls (1*\gs-1.0,-0.9) and (1*\gs-1.0,3*\gs+1.6) .. (1*\gs,3*\gs);
\draw (2*\gs,0) .. controls (2*\gs-1.0,-1.6) and (2*\gs-1.0,3*\gs+0.9) .. (2*\gs,3*\gs);
\draw (3*\gs,0) .. controls (3*\gs-1.0,-0.9) and (3*\gs-1.0,3*\gs+1.6) .. (3*\gs,3*\gs);
\node[fill=white,inner sep=2pt] at (0,3*\gs) {7};
\node[fill=white,inner sep=2pt] at (1*\gs,3*\gs) {12};
\node[fill=white,inner sep=2pt] at (2*\gs,3*\gs) {1};
\node[fill=white,inner sep=2pt] at (3*\gs,3*\gs) {14};
\node[fill=white,inner sep=2pt] at (0,2*\gs) {2};
\node[fill=white,inner sep=2pt] at (1*\gs,2*\gs) {13};
\node[fill=white,inner sep=2pt] at (2*\gs,2*\gs) {8};
\node[fill=white,inner sep=2pt] at (3*\gs,2*\gs) {11};
\node[fill=white,inner sep=2pt] at (0,1*\gs) {16};
\node[fill=white,inner sep=2pt] at (1*\gs,1*\gs) {3};
\node[fill=white,inner sep=2pt] at (2*\gs,1*\gs) {10};
\node[fill=white,inner sep=2pt] at (3*\gs,1*\gs) {5};
\node[fill=white,inner sep=2pt] at (0,0) {9};
\node[fill=white,inner sep=2pt] at (1*\gs,0) {6};
\node[fill=white,inner sep=2pt] at (2*\gs,0) {15};
\node[fill=white,inner sep=2pt] at (3*\gs,0) {4};
\node[text=black] at (3.9,-1.2) {$C_4\times C_4$};
\node[text=black] at (9.6,3.9) {\Huge $\;\;\cong\;\;$};
\begin{scope}[xshift=11.5cm,scale=0.7156]
\draw (0,0) -- (10,0);
\draw (0,0) -- (0,10);
\draw (10,0) -- (10,10);
\draw (0,10) -- (10,10);
\draw (2.5,1.5) -- (8.52,1.5);
\draw (9.08,1.5) -- (9.72,1.5);
\draw (10.28,1.5) -- (12.5,1.5);
\draw (12.5,1.5) -- (12.5,11.5);
\draw (12.5,11.5) -- (2.5,11.5);
\draw (2.5,11.5) -- (2.5,10.28);
\draw (2.5,9.72) -- (2.5,8.306);
\draw (2.5,7.746) -- (2.5,2.912);
\draw (2.5,2.352) -- (2.5,1.5);
\draw (0,0) -- (2.5,1.5);
\draw (10,0) -- (12.5,1.5);
\draw (10,10) -- (12.5,11.5);
\draw (0,10) -- (2.5,11.5);
\draw (3.8,4) -- (6.8,4) -- (6.8,7) -- (3.8,7) -- cycle;
\draw (5.675,5.125) -- (6.52,5.125);
\draw (7.08,5.125) -- (8.675,5.125);
\draw (8.675,5.125) -- (8.675,8.125);
\draw (5.675,8.125) -- (7.72,8.125);
\draw (8.28,8.125) -- (8.675,8.125);
\draw (5.675,8.125) -- (5.675,7.28);
\draw (5.675,6.72) -- (5.675,5.125);
\draw (3.8,4) -- (5.675,5.125);
\draw (6.8,4) -- (8.675,5.125);
\draw (6.8,7) -- (8.675,8.125);
\draw (3.8,7) -- (5.675,8.125);
\draw (0,0) -- (3.8,4);
\draw (10,0) -- (6.8,4);
\draw (10,10) -- (6.8,7);
\draw (0,10) -- (3.8,7);
\draw (2.5,1.5) -- (4.505,3.789);
\draw (4.874,4.211) -- (5.675,5.125);
\draw (12.5,1.5) -- (10.2,3.677);
\draw (9.797,4.062) -- (8.675,5.125);
\draw (12.5,11.5) -- (10.21,9.479);
\draw (9.79,9.109) -- (8.675,8.125);
\draw (2.5,11.5) -- (3.719,10.2);
\draw (4.103,9.796) -- (5.675,8.125);
\node[fill=white,inner sep=2pt] at (0,0) {$1$};
\node[fill=white,inner sep=2pt] at (10,0) {$15$};
\node[fill=white,inner sep=2pt] at (10,10) {$4$};
\node[fill=white,inner sep=2pt] at (0,10) {$14$};
\node[fill=white,inner sep=2pt] at (2.5,1.5) {$12$};
\node[fill=white,inner sep=2pt] at (12.5,1.5) {$6$};
\node[fill=white,inner sep=2pt] at (12.5,11.5) {$9$};
\node[fill=white,inner sep=2pt] at (2.5,11.5) {$7$};
\node[fill=white,inner sep=2pt] at (3.8,4) {$8$};
\node[fill=white,inner sep=2pt] at (6.8,4) {$10$};
\node[fill=white,inner sep=2pt] at (6.8,7) {$5$};
\node[fill=white,inner sep=2pt] at (3.8,7) {$11$};
\node[fill=white,inner sep=2pt] at (5.675,5.125) {$13$};
\node[fill=white,inner sep=2pt] at (8.675,5.125) {$3$};
\node[fill=white,inner sep=2pt] at (8.675,8.125) {$16$};
\node[fill=white,inner sep=2pt] at (5.675,8.125) {$2$};
\end{scope}
\node[text=black] at (15.25,-1.2) {$Q_4$};
\end{tikzpicture}

\caption{An explicit graph isomorphism $C_4\times C_4\cong Q_4$: matching vertex labels correspond under the isomorphism. Left, $C_4\times C_4$ drawn as a toroidal $4\times4$ grid and labelled by the Khajuraho square of Figure~\ref{fig:khajuraho} (row/column wrap-around edges drawn as arcs); right, $Q_4$ drawn exactly as in the $n=4$ case of Figure~\ref{mfh}. This isomorphism carries the Khajuraho square to a magic-faced tesseract.}
\label{fig:isomorphism}
\vspace{-.1in}
\end{figure}
Figure~\ref{fig:isomorphism} makes this isomorphism explicit: the vertices of $C_4\times C_4$, drawn as a $4\times4$ toroidal graph and labelled by the entries of the Khajuraho magic square, correspond to the vertices of $Q_4$ drawn exactly as in the $n=4$ case of Figure~\ref{mfh}. Every edge of one graph maps to an edge of the other, so this labelling exhibits a  graph isomorphism $C_4\times C_4\cong Q_4$ under which the Khajuraho magic square is transformed into the standard magic-faced tesseract in Figure~\ref{mfh}.

By Theorems~\ref{thm:main}--\ref{thm:folded2}, we conclude: 

\begin{theorem}
\label{thm:1920}
There are $2^4\cdot 5!=1920$ magic-faced tesseracts, and hence $1920$ arrangements of $1,\dots,16$ in a $4\times4$ square such that every row, column, and $2\times2$ subsquare---including those that wrap around---sums to the magic constant $34$. These are all equivalent to each other under a natural, faithful action of the group $C_2^4 \rtimes S_{5}$ of order $1920$. 
\end{theorem}

The Khajuraho Square has a further special property, however: every diagonal, and indeed even every \emph{pandiagonal} (i.e., wrap-around, or ``broken'' diagonal) of the $4\times4$ square also sums to the magic constant 34!  In fact, it is known that, for $4\times 4$ magic squares, this ``pandiagonality" property implies another,  stronger property called ``completeness": every pair of  entries two apart on any diagonal or pandiagonal adds up to the same constant 17 (i.e., to half the magic constant). For example, in the Khajuraho Magic Square, note that the (1,1) and (3,3) entries sum to $7+10 = 17$, as do the entries (2,4) and (4,2): $11+6=17$. 

How many $4\times4$ magic squares are there for which, as with the Khajuraho Magic Square, every row, column, $2\times2$ subsquare, diagonal and pandiagonal sums to the same magic constant? Such squares are called \emph{most-perfect} \cite{OllerenshawBree1998}. It is known that there are $384$ of~them. They were all explicitly constructed by Nārāyaṇa Paṇḍita in his \emph{Ga\d nita Kaumud\={\i}} in the year~1356, using his  ``knight's-move algorithm" \cite{Narayana1356,Bhowmik,DuttaSingh2019,SridharanSrinivas2012}. 
The count of $384$ was rigorously proven to be exact by Rosser and Walker~\cite{RosserWalker1938}, who exhibited a group of order $384$ permuting these squares.
Coxeter was intrigued by the number~$384$, recognizing it as the order of the hyperoctahedral group $W(B_4)$, and asked whether this group might be naturally isomorphic to $W(B_4)$.
This was formally verified to be the case by Rosser and Walker~\cite[footnote~p.\,419]{RosserWalker1938}.

Vijayaraghavan~\cite{Vijayaraghavan1941} gave a remarkably simple method to construct all $384$ squares, which made clear that there is a group of order 384 isomorphic to 
$(\mathbb Z/4\mathbb Z)^2\rtimes S_4$ that permutes the squares.  
M\"uller~\cite{Muller1997} first made the presence of a hypercube explicit, labeling the cells of~the $4\times4$ matrix by the vertices of the $4$-cube so that $2\times2$ submatrices correspond to square faces (in hindsight, this is explained by the graph isomorphism $C_4\times C_4\cong Q_4$); this showed that $W(B_4)$ acts simply transitively on the $384$ squares. 
Finally, Navas~\cite{Navas2020}~gave~a beautiful geometric proof that the symmetry group of~the $384$ pandiagonal squares is that of the hyper\-cube, by showing that the natural distance on the $4\times4$ board (the number of cell boundaries crossed) becomes the hypercube distance once its cells are suitably indexed by~$\{0,1\}^4$.

\pagebreak
%
%
The viewpoint of the exceptional graph isomorphism $C_4\times C_4\cong Q_4$ enables us to unify the above perspectives, making the natural group action of $W(B_4)$ conceptually evident. As~we shall see, this viewpoint also enables us to generalize the story to all dimensions. 

We~first observe that not all $1920$ magic-faced tesseracts enumerated in Theorem~\ref{thm:1920} lie in a single orbit for the action of $\Aut(Q_4)$.  We prove that the magic tesseracts of Theorem~\ref{thm:1920} that yield pandiagonal magic squares under this exceptional graph isomorphism are precisely those that are {\bf associative}. By Theorem~\ref{thm:associative}, the symmetries that preserve the condition of being magic-faced and associative form the subgroup $\Aut(Q_4)\subset \Aut(FQ_4)$, i.e., the symmetry group $\Aut(Q_4)\cong W(B_4)$ of the tesseract. 

This yields the desired conceptual answer to Coxeter's question: we view a most-perfect $4\times4$ magic square as an associative magic-faced tesseract, i.e., as a labelling of $Q_4$ where every $4$-cycle has the same sum and every pair of antipodal entries has the same sum. Theorem~\ref{thm:associative} states that there is only one such magic-faced tesseract up to the action of the symmetry group $\Aut(Q_4)\cong W(B_4)$ of the tesseract $Q_4$. Hence there are $|\Aut(Q_4)| = |W(B_4)| = 384$ such squares, naturally answering Coxeter's question in the positive and giving a clear explanation for the number $384$ of most-perfect $4\times4$ magic squares. 

\begin{theorem}
\label{thm:384orbit}
Under the graph isomorphism $C_4\times C_4 \xrightarrow{\sim} Q_4$, the $384$ most-perfect $4\times4$ magic squares form the single orbit of the image of the Khajuraho Magic Square under the faithful action of the group $\Aut_{\mathrm{graph}}(C_4\times C_4)=\Aut(Q_4)\cong W(B_4) = C_2^4\rtimes S_4$ of order $16\times 24=384$.
\end{theorem}
\noindent
That is, the Khajuraho Magic Square is the unique most-perfect $4\times 4$ magic square up to the natural, faithful action of the group $\Aut_{\mathrm{graph}}(C_4\times C_4)=\Aut(Q_4)\cong W(B_4)$. 

\vspace{.025in}

\subsection*{\normalsize Most-perfect Hypercubes in Higher Dimensions, and the Analogue of Coxeter's Question}


\vspace{.05in}
Our perspective enables the proof of existence, and indeed the classification and enumeration, of most-perfect magic hypercubes of order 4 in every dimension $m\geq 1$.  
A {\it most-perfect magic hypercube of order $4$ and dimension $m$} is an arrangement of the numbers $1,2,\ldots,4^m$ in a $4\times\cdots \times 4$ ($m$-fold) grid such that every orthogonal line has the same magic sum (i.e., it is {\it magic}), every $2\times 2$ subsquare also has that same magic sum (i.e., it is {\it compact}), every great diagonal and great pandiagonal also has that same magic sum (i.e., it is {\it pandiagonal}), and every pair of elements two entries apart on such a pandiagonal adds up to half the magic~sum (i.e., it is {\it complete}).  It~had not been known previously whether most-perfect magic hypercubes exist
in dimensions greater than~3.  

We prove that most-perfect magic hypercubes of order $4$ exist in every dimension~$m$, and indeed we classify all of them. 
The answer turns out to be quite beautiful: generalizing Coxeter's 1936~query, we prove that, for each $m$, most-perfect hypercubes of order~4 and dimension~$m$ form one  orbit under a natural and faithful action of the Weyl group~$W(B_{2m})$. 

In fact, we prove something more general that explains the structure of these objects more fully.  We classify and enumerate the magic hypercubes of order 4 in each dimension that are: compact; compact and pandiagonal, and compact and complete, i.e., most-perfect. 

We prove the following theorems:

\pagebreak 

\begin{theorem}
\label{thm:final1}
There exists a unique compact magic hypercube of order $4$ and dimension~$m$ $($i.e., a magic hypercube of order~$4$ and dimension~$m$ with entries $1,\dots,4^m$ such that all orthogonal lines and all $2\times2$ subsquares, including wrap-around $2\times 2$ subsquares, sum~to~the same magic constant$)$ up to the natural faithful action of the group $\Aut(FQ_{2m})\!\cong\! C_2^{2m} \rtimes S_{2m+1}.$ 
\end{theorem}

\begin{theorem} \label{thm:final2}
There exists a unique most-perfect magic hypercube of order $4$ and dimension~$m$ $($i.e., a magic hypercube of order~$4$ and dimension $m$ with entries $1,\dots,4^m$ such that all orthogonal lines, all great diagonals and  pandiagonals, all $2\times2$ subsquares, including wrap-around $2\times 2$ subsquares, sum to the same magic constant; and all pairs of entries lying two apart on any great pandiagonal sum to half the magic constant$)$, up to the natural, faithful action of the group $W(B_{2m}) \cong \Aut(Q_{2m})\cong \Aut_{\mathrm{graph}}(C_4^m)\cong C_2^{2m}\rtimes S_{2m}.$
\end{theorem} 
\noindent
Theorems~\ref{thm:final1} and \ref{thm:final2} form the generalizations of Theorems~\ref{thm:1920} and \ref{thm:384orbit} to arbitrary  dimensions, and Theorem~\ref{thm:final2} answers the analogue of Coxeter's question for general dimensions~$m$. 

Theorems~\ref{thm:final1} and \ref{thm:final2} not only show existence but also give us the exact count of compact and of most-perfect  magic hypercubes of order $4$ in each dimension~$m$. 

\begin{corollary}\label{thm:countmostperfect}
\phantom{hello.}
\vspace{-.055in}

\begin{itemize}
\item[{\em (a)}] 
The number of 
compact magic hypercubes of order $4$ and dimension $m$ is $4^m(2m\!+\!1)!$. 

\vspace{-.025in}
\item[{\em (b)}] 
The number of 
most-perfect magic hypercubes of order $4$ and dimension $m$ is 
$4^m(2m)!$. 
\end{itemize}
\end{corollary}

We note that Theorems~\ref{thm:final1}--\ref{thm:final2} and Corollary~\ref{thm:countmostperfect} also classify and enumerate the magic hypercubes of order 4 that are compact and pandiagonal, due to the following theorem. 

\begin{theorem}\label{thm:implications}
If $m$ is odd, then a compact magic hypercube of order $4$ is  automatically pandiagonal.  If $m$ is even, then a compact and pandiagonal magic hypercube of order $4$ is automatically most-perfect. 
\end{theorem}
\noindent
The classification and enumeration of compact and pandiagonal magic hypercubes of order~$4$ and dimension $m$ are therefore covered by  Theorem \ref{thm:final1} and Corollary~\ref{thm:countmostperfect}(a) when $m$ is odd, and by Theorem~\ref{thm:final2} and Corollary~\ref{thm:countmostperfect}(b) when $m$ is even. 

As  mentioned earlier, the existence of most-perfect hypercubes in dimension $2$ was demonstrated by the Khajuraho Magic Square. The 384 such most-perfect magic squares of order~4, as studied by Nārāyaṇa Paṇḍita~\cite{Narayana1356} and Rosser and Walker~\cite{RosserWalker1938,RosserWalker1939}---and the corresponding question of Coxeter---are fully explained by the $m=2$ case of Theorem~\ref{thm:final2} and in turn by the $n=4$ case of Theorem~\ref{thm:associative}.  

In dimension $3$, an example of a most-perfect magic cube of order $4$ can be found amongst the extensive computations of the modern magic-hypercube maestro Hendricks~\cite{Hendricks1972}, who constructed by hand 7860 pandiagonal magic cubes in an attempt to classify all such objects.  Of these, we noticed that some are actually compact and complete (as we knew they must be by Theorem~\ref{thm:final2}), though it seems this was not originally noticed or checked by Hendricks at the time, as this would have been an additional rather large search and computation by hand. Hendricks~\cite{Hendricks1968} also constructed a number of pandiagonal magic tesseracts, but none of them were fully compact and complete. 

\pagebreak

In 2016, Knecht~\cite{oeis} performed an extensive computer calculation to find 46,080 most-perfect magic cubes, and he conjectured that these are all most-perfect cubes of order 4.  Since $|W(B_{6})|=2^6\cdot 6! = 46080$, Theorem~\ref{thm:final2} fully explains Knecht's computation and conjecture for $m=3$ group-theoretically.  

Around the same time, Trump~\cite{oeis2} conjectured a number of properties that he expected to be true for a hypothetical most-perfect magic tesseract of order $4$, based on observed properties in dimensions 2 and 3. For example, he predicted what entries should be adjacent to the entry 1 in such a hypothetical tesseract.  The arguments in the proof of Theorem~\ref{thm:final2} 
also confirm Trump's conjectures for a most-perfect magic tesseract. 

\begin{figure}[htpb]
\centering
\begin{tikzpicture}[every path/.style={draw=red}, every node/.style={text=blue}, scale=0.42]
\draw (0,0) -- (10,0);
\draw (10,0) -- (20,0);
\draw (20,0) -- (30,0);
\draw (0,10) -- (10,10);
\draw (10,10) -- (20,10);
\draw (20,10) -- (30,10);
\draw (0,20) -- (10,20);
\draw (10,20) -- (20,20);
\draw (20,20) -- (30,20);
\draw (0,30) -- (10,30);
\draw (10,30) -- (20,30);
\draw (20,30) -- (30,30);
\draw (0,0) -- (0,10);
\draw (0,10) -- (0,20);
\draw (0,20) -- (0,30);
\draw (10,0) -- (10,10);
\draw (10,10) -- (10,20);
\draw (10,20) -- (10,30);
\draw (20,0) -- (20,10);
\draw (20,10) -- (20,20);
\draw (20,20) -- (20,30);
\draw (30,0) -- (30,10);
\draw (30,10) -- (30,20);
\draw (30,20) -- (30,30);
\draw (1.6,1.6) -- (9.78,1.6);
\draw (10.22,1.6) -- (11.6,1.6);
\draw (11.6,1.6) -- (19.78,1.6);
\draw (20.22,1.6) -- (21.6,1.6);
\draw (21.6,1.6) -- (29.78,1.6);
\draw (30.22,1.6) -- (31.6,1.6);
\draw (1.6,11.6) -- (9.78,11.6);
\draw (10.22,11.6) -- (11.6,11.6);
\draw (11.6,11.6) -- (19.78,11.6);
\draw (20.22,11.6) -- (21.6,11.6);
\draw (21.6,11.6) -- (29.78,11.6);
\draw (30.22,11.6) -- (31.6,11.6);
\draw (1.6,21.6) -- (9.78,21.6);
\draw (10.22,21.6) -- (11.6,21.6);
\draw (11.6,21.6) -- (19.78,21.6);
\draw (20.22,21.6) -- (21.6,21.6);
\draw (21.6,21.6) -- (29.78,21.6);
\draw (30.22,21.6) -- (31.6,21.6);
\draw (1.6,31.6) -- (11.6,31.6);
\draw (11.6,31.6) -- (21.6,31.6);
\draw (21.6,31.6) -- (31.6,31.6);
\draw (1.6,1.6) -- (1.6,9.78);
\draw (1.6,10.22) -- (1.6,11.6);
\draw (1.6,11.6) -- (1.6,19.78);
\draw (1.6,20.22) -- (1.6,21.6);
\draw (1.6,21.6) -- (1.6,29.78);
\draw (1.6,30.22) -- (1.6,31.6);
\draw (11.6,1.6) -- (11.6,9.78);
\draw (11.6,10.22) -- (11.6,11.6);
\draw (11.6,11.6) -- (11.6,19.78);
\draw (11.6,20.22) -- (11.6,21.6);
\draw (11.6,21.6) -- (11.6,29.78);
\draw (11.6,30.22) -- (11.6,31.6);
\draw (21.6,1.6) -- (21.6,9.78);
\draw (21.6,10.22) -- (21.6,11.6);
\draw (21.6,11.6) -- (21.6,19.78);
\draw (21.6,20.22) -- (21.6,21.6);
\draw (21.6,21.6) -- (21.6,29.78);
\draw (21.6,30.22) -- (21.6,31.6);
\draw (31.6,1.6) -- (31.6,11.6);
\draw (31.6,11.6) -- (31.6,21.6);
\draw (31.6,21.6) -- (31.6,31.6);
\draw (3.2,3.2) -- (9.78,3.2);
\draw (10.22,3.2) -- (11.38,3.2);
\draw (11.82,3.2) -- (13.2,3.2);
\draw (13.2,3.2) -- (19.78,3.2);
\draw (20.22,3.2) -- (21.38,3.2);
\draw (21.82,3.2) -- (23.2,3.2);
\draw (23.2,3.2) -- (29.78,3.2);
\draw (30.22,3.2) -- (31.38,3.2);
\draw (31.82,3.2) -- (33.2,3.2);
\draw (3.2,13.2) -- (9.78,13.2);
\draw (10.22,13.2) -- (11.38,13.2);
\draw (11.82,13.2) -- (13.2,13.2);
\draw (13.2,13.2) -- (19.78,13.2);
\draw (20.22,13.2) -- (21.38,13.2);
\draw (21.82,13.2) -- (23.2,13.2);
\draw (23.2,13.2) -- (29.78,13.2);
\draw (30.22,13.2) -- (31.38,13.2);
\draw (31.82,13.2) -- (33.2,13.2);
\draw (3.2,23.2) -- (9.78,23.2);
\draw (10.22,23.2) -- (11.38,23.2);
\draw (11.82,23.2) -- (13.2,23.2);
\draw (13.2,23.2) -- (19.78,23.2);
\draw (20.22,23.2) -- (21.38,23.2);
\draw (21.82,23.2) -- (23.2,23.2);
\draw (23.2,23.2) -- (29.78,23.2);
\draw (30.22,23.2) -- (31.38,23.2);
\draw (31.82,23.2) -- (33.2,23.2);
\draw (3.2,33.2) -- (13.2,33.2);
\draw (13.2,33.2) -- (23.2,33.2);
\draw (23.2,33.2) -- (33.2,33.2);
\draw (3.2,3.2) -- (3.2,9.78);
\draw (3.2,10.22) -- (3.2,11.38);
\draw (3.2,11.82) -- (3.2,13.2);
\draw (3.2,13.2) -- (3.2,19.78);
\draw (3.2,20.22) -- (3.2,21.38);
\draw (3.2,21.82) -- (3.2,23.2);
\draw (3.2,23.2) -- (3.2,29.78);
\draw (3.2,30.22) -- (3.2,31.38);
\draw (3.2,31.82) -- (3.2,33.2);
\draw (13.2,3.2) -- (13.2,9.78);
\draw (13.2,10.22) -- (13.2,11.38);
\draw (13.2,11.82) -- (13.2,13.2);
\draw (13.2,13.2) -- (13.2,19.78);
\draw (13.2,20.22) -- (13.2,21.38);
\draw (13.2,21.82) -- (13.2,23.2);
\draw (13.2,23.2) -- (13.2,29.78);
\draw (13.2,30.22) -- (13.2,31.38);
\draw (13.2,31.82) -- (13.2,33.2);
\draw (23.2,3.2) -- (23.2,9.78);
\draw (23.2,10.22) -- (23.2,11.38);
\draw (23.2,11.82) -- (23.2,13.2);
\draw (23.2,13.2) -- (23.2,19.78);
\draw (23.2,20.22) -- (23.2,21.38);
\draw (23.2,21.82) -- (23.2,23.2);
\draw (23.2,23.2) -- (23.2,29.78);
\draw (23.2,30.22) -- (23.2,31.38);
\draw (23.2,31.82) -- (23.2,33.2);
\draw (33.2,3.2) -- (33.2,13.2);
\draw (33.2,13.2) -- (33.2,23.2);
\draw (33.2,23.2) -- (33.2,33.2);
\draw (4.8,4.8) -- (9.78,4.8);
\draw (10.22,4.8) -- (11.38,4.8);
\draw (11.82,4.8) -- (12.98,4.8);
\draw (13.42,4.8) -- (14.8,4.8);
\draw (14.8,4.8) -- (19.78,4.8);
\draw (20.22,4.8) -- (21.38,4.8);
\draw (21.82,4.8) -- (22.98,4.8);
\draw (23.42,4.8) -- (24.8,4.8);
\draw (24.8,4.8) -- (29.78,4.8);
\draw (30.22,4.8) -- (31.38,4.8);
\draw (31.82,4.8) -- (32.98,4.8);
\draw (33.42,4.8) -- (34.8,4.8);
\draw (4.8,14.8) -- (9.78,14.8);
\draw (10.22,14.8) -- (11.38,14.8);
\draw (11.82,14.8) -- (12.98,14.8);
\draw (13.42,14.8) -- (14.8,14.8);
\draw (14.8,14.8) -- (19.78,14.8);
\draw (20.22,14.8) -- (21.38,14.8);
\draw (21.82,14.8) -- (22.98,14.8);
\draw (23.42,14.8) -- (24.8,14.8);
\draw (24.8,14.8) -- (29.78,14.8);
\draw (30.22,14.8) -- (31.38,14.8);
\draw (31.82,14.8) -- (32.98,14.8);
\draw (33.42,14.8) -- (34.8,14.8);
\draw (4.8,24.8) -- (9.78,24.8);
\draw (10.22,24.8) -- (11.38,24.8);
\draw (11.82,24.8) -- (12.98,24.8);
\draw (13.42,24.8) -- (14.8,24.8);
\draw (14.8,24.8) -- (19.78,24.8);
\draw (20.22,24.8) -- (21.38,24.8);
\draw (21.82,24.8) -- (22.98,24.8);
\draw (23.42,24.8) -- (24.8,24.8);
\draw (24.8,24.8) -- (29.78,24.8);
\draw (30.22,24.8) -- (31.38,24.8);
\draw (31.82,24.8) -- (32.98,24.8);
\draw (33.42,24.8) -- (34.8,24.8);
\draw (4.8,34.8) -- (14.8,34.8);
\draw (14.8,34.8) -- (24.8,34.8);
\draw (24.8,34.8) -- (34.8,34.8);
\draw (4.8,4.8) -- (4.8,9.78);
\draw (4.8,10.22) -- (4.8,11.38);
\draw (4.8,11.82) -- (4.8,12.98);
\draw (4.8,13.42) -- (4.8,14.8);
\draw (4.8,14.8) -- (4.8,19.78);
\draw (4.8,20.22) -- (4.8,21.38);
\draw (4.8,21.82) -- (4.8,22.98);
\draw (4.8,23.42) -- (4.8,24.8);
\draw (4.8,24.8) -- (4.8,29.78);
\draw (4.8,30.22) -- (4.8,31.38);
\draw (4.8,31.82) -- (4.8,32.98);
\draw (4.8,33.42) -- (4.8,34.8);
\draw (14.8,4.8) -- (14.8,9.78);
\draw (14.8,10.22) -- (14.8,11.38);
\draw (14.8,11.82) -- (14.8,12.98);
\draw (14.8,13.42) -- (14.8,14.8);
\draw (14.8,14.8) -- (14.8,19.78);
\draw (14.8,20.22) -- (14.8,21.38);
\draw (14.8,21.82) -- (14.8,22.98);
\draw (14.8,23.42) -- (14.8,24.8);
\draw (14.8,24.8) -- (14.8,29.78);
\draw (14.8,30.22) -- (14.8,31.38);
\draw (14.8,31.82) -- (14.8,32.98);
\draw (14.8,33.42) -- (14.8,34.8);
\draw (24.8,4.8) -- (24.8,9.78);
\draw (24.8,10.22) -- (24.8,11.38);
\draw (24.8,11.82) -- (24.8,12.98);
\draw (24.8,13.42) -- (24.8,14.8);
\draw (24.8,14.8) -- (24.8,19.78);
\draw (24.8,20.22) -- (24.8,21.38);
\draw (24.8,21.82) -- (24.8,22.98);
\draw (24.8,23.42) -- (24.8,24.8);
\draw (24.8,24.8) -- (24.8,29.78);
\draw (24.8,30.22) -- (24.8,31.38);
\draw (24.8,31.82) -- (24.8,32.98);
\draw (24.8,33.42) -- (24.8,34.8);
\draw (34.8,4.8) -- (34.8,14.8);
\draw (34.8,14.8) -- (34.8,24.8);
\draw (34.8,24.8) -- (34.8,34.8);
\draw (0,0) -- (1.6,1.6);
\draw (1.6,1.6) -- (3.2,3.2);
\draw (3.2,3.2) -- (4.8,4.8);
\draw (0,10) -- (1.6,11.6);
\draw (1.6,11.6) -- (3.2,13.2);
\draw (3.2,13.2) -- (4.8,14.8);
\draw (0,20) -- (1.6,21.6);
\draw (1.6,21.6) -- (3.2,23.2);
\draw (3.2,23.2) -- (4.8,24.8);
\draw (0,30) -- (1.6,31.6);
\draw (1.6,31.6) -- (3.2,33.2);
\draw (3.2,33.2) -- (4.8,34.8);
\draw (10,0) -- (11.6,1.6);
\draw (11.6,1.6) -- (13.2,3.2);
\draw (13.2,3.2) -- (14.8,4.8);
\draw (10,10) -- (11.6,11.6);
\draw (11.6,11.6) -- (13.2,13.2);
\draw (13.2,13.2) -- (14.8,14.8);
\draw (10,20) -- (11.6,21.6);
\draw (11.6,21.6) -- (13.2,23.2);
\draw (13.2,23.2) -- (14.8,24.8);
\draw (10,30) -- (11.6,31.6);
\draw (11.6,31.6) -- (13.2,33.2);
\draw (13.2,33.2) -- (14.8,34.8);
\draw (20,0) -- (21.6,1.6);
\draw (21.6,1.6) -- (23.2,3.2);
\draw (23.2,3.2) -- (24.8,4.8);
\draw (20,10) -- (21.6,11.6);
\draw (21.6,11.6) -- (23.2,13.2);
\draw (23.2,13.2) -- (24.8,14.8);
\draw (20,20) -- (21.6,21.6);
\draw (21.6,21.6) -- (23.2,23.2);
\draw (23.2,23.2) -- (24.8,24.8);
\draw (20,30) -- (21.6,31.6);
\draw (21.6,31.6) -- (23.2,33.2);
\draw (23.2,33.2) -- (24.8,34.8);
\draw (30,0) -- (31.6,1.6);
\draw (31.6,1.6) -- (33.2,3.2);
\draw (33.2,3.2) -- (34.8,4.8);
\draw (30,10) -- (31.6,11.6);
\draw (31.6,11.6) -- (33.2,13.2);
\draw (33.2,13.2) -- (34.8,14.8);
\draw (30,20) -- (31.6,21.6);
\draw (31.6,21.6) -- (33.2,23.2);
\draw (33.2,23.2) -- (34.8,24.8);
\draw (30,30) -- (31.6,31.6);
\draw (31.6,31.6) -- (33.2,33.2);
\draw (33.2,33.2) -- (34.8,34.8);
\node[fill=white,inner sep=1pt] at (0,0) {$1$};
\node[fill=white,inner sep=1pt] at (10,0) {$62$};
\node[fill=white,inner sep=1pt] at (20,0) {$4$};
\node[fill=white,inner sep=1pt] at (30,0) {$63$};
\node[fill=white,inner sep=1pt] at (0,10) {$56$};
\node[fill=white,inner sep=1pt] at (10,10) {$11$};
\node[fill=white,inner sep=1pt] at (20,10) {$53$};
\node[fill=white,inner sep=1pt] at (30,10) {$10$};
\node[fill=white,inner sep=1pt] at (0,20) {$13$};
\node[fill=white,inner sep=1pt] at (10,20) {$50$};
\node[fill=white,inner sep=1pt] at (20,20) {$16$};
\node[fill=white,inner sep=1pt] at (30,20) {$51$};
\node[fill=white,inner sep=1pt] at (0,30) {$60$};
\node[fill=white,inner sep=1pt] at (10,30) {$7$};
\node[fill=white,inner sep=1pt] at (20,30) {$57$};
\node[fill=white,inner sep=1pt] at (30,30) {$6$};
\node[fill=white,inner sep=1pt] at (1.6,1.6) {$32$};
\node[fill=white,inner sep=1pt] at (11.6,1.6) {$35$};
\node[fill=white,inner sep=1pt] at (21.6,1.6) {$29$};
\node[fill=white,inner sep=1pt] at (31.6,1.6) {$34$};
\node[fill=white,inner sep=1pt] at (1.6,11.6) {$41$};
\node[fill=white,inner sep=1pt] at (11.6,11.6) {$22$};
\node[fill=white,inner sep=1pt] at (21.6,11.6) {$44$};
\node[fill=white,inner sep=1pt] at (31.6,11.6) {$23$};
\node[fill=white,inner sep=1pt] at (1.6,21.6) {$20$};
\node[fill=white,inner sep=1pt] at (11.6,21.6) {$47$};
\node[fill=white,inner sep=1pt] at (21.6,21.6) {$17$};
\node[fill=white,inner sep=1pt] at (31.6,21.6) {$46$};
\node[fill=white,inner sep=1pt] at (1.6,31.6) {$37$};
\node[fill=white,inner sep=1pt] at (11.6,31.6) {$26$};
\node[fill=white,inner sep=1pt] at (21.6,31.6) {$40$};
\node[fill=white,inner sep=1pt] at (31.6,31.6) {$27$};
\node[fill=white,inner sep=1pt] at (3.2,3.2) {$49$};
\node[fill=white,inner sep=1pt] at (13.2,3.2) {$14$};
\node[fill=white,inner sep=1pt] at (23.2,3.2) {$52$};
\node[fill=white,inner sep=1pt] at (33.2,3.2) {$15$};
\node[fill=white,inner sep=1pt] at (3.2,13.2) {$8$};
\node[fill=white,inner sep=1pt] at (13.2,13.2) {$59$};
\node[fill=white,inner sep=1pt] at (23.2,13.2) {$5$};
\node[fill=white,inner sep=1pt] at (33.2,13.2) {$58$};
\node[fill=white,inner sep=1pt] at (3.2,23.2) {$61$};
\node[fill=white,inner sep=1pt] at (13.2,23.2) {$2$};
\node[fill=white,inner sep=1pt] at (23.2,23.2) {$64$};
\node[fill=white,inner sep=1pt] at (33.2,23.2) {$3$};
\node[fill=white,inner sep=1pt] at (3.2,33.2) {$12$};
\node[fill=white,inner sep=1pt] at (13.2,33.2) {$55$};
\node[fill=white,inner sep=1pt] at (23.2,33.2) {$9$};
\node[fill=white,inner sep=1pt] at (33.2,33.2) {$54$};
\node[fill=white,inner sep=1pt] at (4.8,4.8) {$48$};
\node[fill=white,inner sep=1pt] at (14.8,4.8) {$19$};
\node[fill=white,inner sep=1pt] at (24.8,4.8) {$45$};
\node[fill=white,inner sep=1pt] at (34.8,4.8) {$18$};
\node[fill=white,inner sep=1pt] at (4.8,14.8) {$25$};
\node[fill=white,inner sep=1pt] at (14.8,14.8) {$38$};
\node[fill=white,inner sep=1pt] at (24.8,14.8) {$28$};
\node[fill=white,inner sep=1pt] at (34.8,14.8) {$39$};
\node[fill=white,inner sep=1pt] at (4.8,24.8) {$36$};
\node[fill=white,inner sep=1pt] at (14.8,24.8) {$31$};
\node[fill=white,inner sep=1pt] at (24.8,24.8) {$33$};
\node[fill=white,inner sep=1pt] at (34.8,24.8) {$30$};
\node[fill=white,inner sep=1pt] at (4.8,34.8) {$21$};
\node[fill=white,inner sep=1pt] at (14.8,34.8) {$42$};
\node[fill=white,inner sep=1pt] at (24.8,34.8) {$24$};
\node[fill=white,inner sep=1pt] at (34.8,34.8) {$43$};
\end{tikzpicture}
\vspace{.15in}
\caption{A $4\times4\times4$ most-perfect magic cube with entries $1,\dots,64$. Every orthogonal line, every $2\times2$ subsquare, and every great diagonal and {\it great pandiagonal}---i.e., every toroidal line in a great diagonal direction $(\pm1,\pm1,\pm1)$---sums to the magic constant~$130$.  In addition, every pair of entries two apart on any great pandiagonal sums to 65---half the magic \vspace{.125in} constant. 
\linebreak 
This  Magic Cube thus has all the analogous properties in three dimensions that the Khajuraho Magic Square has in two dimensions. By Theorem~\ref{thm:final2}, 
all such most-perfect magic cubes lie in a single orbit for a faithful action of $W(B_6)$. Examples of such most-perfect magic cubes were first constructed by Hendricks and Trump.}
\label{fig:cube64}
\end{figure}

\begin{figure}[htbp]
\centering
\begin{tikzpicture}[scale=0.62]
\draw[red] (-0.55,-0.55) -- (-0.55,22.4);
\draw[red] (-0.55,-0.55) -- (22.4,-0.55);
\draw[red] (5.15,-0.55) -- (5.15,22.4);
\draw[red] (-0.55,5.15) -- (22.4,5.15);
\draw[red] (10.9,-0.55) -- (10.9,22.4);
\draw[red] (-0.55,10.9) -- (22.4,10.9);
\draw[red] (16.8,-0.55) -- (16.8,22.4);
\draw[red] (-0.55,16.8) -- (22.4,16.8);
\draw[red] (22.4,-0.55) -- (22.4,22.4);
\draw[red] (-0.55,22.4) -- (22.4,22.4);
\node[text=blue] at (0.0, 0.0) {$1$};
\node[text=blue] at (0.0, 1.5) {$248$};
\node[text=blue] at (0.0, 3.0) {$13$};
\node[text=blue] at (0.0, 4.5) {$252$};
\node[text=blue] at (1.5, 0.0) {$254$};
\node[text=blue] at (1.5, 1.5) {$11$};
\node[text=blue] at (1.5, 3.0) {$242$};
\node[text=blue] at (1.5, 4.5) {$7$};
\node[text=blue] at (3.0, 0.0) {$4$};
\node[text=blue] at (3.0, 1.5) {$245$};
\node[text=blue] at (3.0, 3.0) {$16$};
\node[text=blue] at (3.0, 4.5) {$249$};
\node[text=blue] at (4.5, 0.0) {$255$};
\node[text=blue] at (4.5, 1.5) {$10$};
\node[text=blue] at (4.5, 3.0) {$243$};
\node[text=blue] at (4.5, 4.5) {$6$};
\node[text=blue] at (0.0, 5.8) {$128$};
\node[text=blue] at (0.0, 7.3) {$137$};
\node[text=blue] at (0.0, 8.8) {$116$};
\node[text=blue] at (0.0, 10.3) {$133$};
\node[text=blue] at (1.5, 5.8) {$131$};
\node[text=blue] at (1.5, 7.3) {$118$};
\node[text=blue] at (1.5, 8.8) {$143$};
\node[text=blue] at (1.5, 10.3) {$122$};
\node[text=blue] at (3.0, 5.8) {$125$};
\node[text=blue] at (3.0, 7.3) {$140$};
\node[text=blue] at (3.0, 8.8) {$113$};
\node[text=blue] at (3.0, 10.3) {$136$};
\node[text=blue] at (4.5, 5.8) {$130$};
\node[text=blue] at (4.5, 7.3) {$119$};
\node[text=blue] at (4.5, 8.8) {$142$};
\node[text=blue] at (4.5, 10.3) {$123$};
\node[text=blue] at (0.0, 11.6) {$193$};
\node[text=blue] at (0.0, 13.1) {$56$};
\node[text=blue] at (0.0, 14.6) {$205$};
\node[text=blue] at (0.0, 16.1) {$60$};
\node[text=blue] at (1.5, 11.6) {$62$};
\node[text=blue] at (1.5, 13.1) {$203$};
\node[text=blue] at (1.5, 14.6) {$50$};
\node[text=blue] at (1.5, 16.1) {$199$};
\node[text=blue] at (3.0, 11.6) {$196$};
\node[text=blue] at (3.0, 13.1) {$53$};
\node[text=blue] at (3.0, 14.6) {$208$};
\node[text=blue] at (3.0, 16.1) {$57$};
\node[text=blue] at (4.5, 11.6) {$63$};
\node[text=blue] at (4.5, 13.1) {$202$};
\node[text=blue] at (4.5, 14.6) {$51$};
\node[text=blue] at (4.5, 16.1) {$198$};
\node[text=blue] at (0.0, 17.4) {$192$};
\node[text=blue] at (0.0, 18.9) {$73$};
\node[text=blue] at (0.0, 20.4) {$180$};
\node[text=blue] at (0.0, 21.9) {$69$};
\node[text=blue] at (1.5, 17.4) {$67$};
\node[text=blue] at (1.5, 18.9) {$182$};
\node[text=blue] at (1.5, 20.4) {$79$};
\node[text=blue] at (1.5, 21.9) {$186$};
\node[text=blue] at (3.0, 17.4) {$189$};
\node[text=blue] at (3.0, 18.9) {$76$};
\node[text=blue] at (3.0, 20.4) {$177$};
\node[text=blue] at (3.0, 21.9) {$72$};
\node[text=blue] at (4.5, 17.4) {$66$};
\node[text=blue] at (4.5, 18.9) {$183$};
\node[text=blue] at (4.5, 20.4) {$78$};
\node[text=blue] at (4.5, 21.9) {$187$};
\node[text=blue] at (5.8, 0.0) {$224$};
\node[text=blue] at (5.8, 1.5) {$41$};
\node[text=blue] at (5.8, 3.0) {$212$};
\node[text=blue] at (5.8, 4.5) {$37$};
\node[text=blue] at (7.3, 0.0) {$35$};
\node[text=blue] at (7.3, 1.5) {$214$};
\node[text=blue] at (7.3, 3.0) {$47$};
\node[text=blue] at (7.3, 4.5) {$218$};
\node[text=blue] at (8.8, 0.0) {$221$};
\node[text=blue] at (8.8, 1.5) {$44$};
\node[text=blue] at (8.8, 3.0) {$209$};
\node[text=blue] at (8.8, 4.5) {$40$};
\node[text=blue] at (10.3, 0.0) {$34$};
\node[text=blue] at (10.3, 1.5) {$215$};
\node[text=blue] at (10.3, 3.0) {$46$};
\node[text=blue] at (10.3, 4.5) {$219$};
\node[text=blue] at (5.8, 5.8) {$161$};
\node[text=blue] at (5.8, 7.3) {$88$};
\node[text=blue] at (5.8, 8.8) {$173$};
\node[text=blue] at (5.8, 10.3) {$92$};
\node[text=blue] at (7.3, 5.8) {$94$};
\node[text=blue] at (7.3, 7.3) {$171$};
\node[text=blue] at (7.3, 8.8) {$82$};
\node[text=blue] at (7.3, 10.3) {$167$};
\node[text=blue] at (8.8, 5.8) {$164$};
\node[text=blue] at (8.8, 7.3) {$85$};
\node[text=blue] at (8.8, 8.8) {$176$};
\node[text=blue] at (8.8, 10.3) {$89$};
\node[text=blue] at (10.3, 5.8) {$95$};
\node[text=blue] at (10.3, 7.3) {$170$};
\node[text=blue] at (10.3, 8.8) {$83$};
\node[text=blue] at (10.3, 10.3) {$166$};
\node[text=blue] at (5.8, 11.6) {$32$};
\node[text=blue] at (5.8, 13.1) {$233$};
\node[text=blue] at (5.8, 14.6) {$20$};
\node[text=blue] at (5.8, 16.1) {$229$};
\node[text=blue] at (7.3, 11.6) {$227$};
\node[text=blue] at (7.3, 13.1) {$22$};
\node[text=blue] at (7.3, 14.6) {$239$};
\node[text=blue] at (7.3, 16.1) {$26$};
\node[text=blue] at (8.8, 11.6) {$29$};
\node[text=blue] at (8.8, 13.1) {$236$};
\node[text=blue] at (8.8, 14.6) {$17$};
\node[text=blue] at (8.8, 16.1) {$232$};
\node[text=blue] at (10.3, 11.6) {$226$};
\node[text=blue] at (10.3, 13.1) {$23$};
\node[text=blue] at (10.3, 14.6) {$238$};
\node[text=blue] at (10.3, 16.1) {$27$};
\node[text=blue] at (5.8, 17.4) {$97$};
\node[text=blue] at (5.8, 18.9) {$152$};
\node[text=blue] at (5.8, 20.4) {$109$};
\node[text=blue] at (5.8, 21.9) {$156$};
\node[text=blue] at (7.3, 17.4) {$158$};
\node[text=blue] at (7.3, 18.9) {$107$};
\node[text=blue] at (7.3, 20.4) {$146$};
\node[text=blue] at (7.3, 21.9) {$103$};
\node[text=blue] at (8.8, 17.4) {$100$};
\node[text=blue] at (8.8, 18.9) {$149$};
\node[text=blue] at (8.8, 20.4) {$112$};
\node[text=blue] at (8.8, 21.9) {$153$};
\node[text=blue] at (10.3, 17.4) {$159$};
\node[text=blue] at (10.3, 18.9) {$106$};
\node[text=blue] at (10.3, 20.4) {$147$};
\node[text=blue] at (10.3, 21.9) {$102$};
\node[text=blue] at (11.6, 0.0) {$49$};
\node[text=blue] at (11.6, 1.5) {$200$};
\node[text=blue] at (11.6, 3.0) {$61$};
\node[text=blue] at (11.6, 4.5) {$204$};
\node[text=blue] at (13.1, 0.0) {$206$};
\node[text=blue] at (13.1, 1.5) {$59$};
\node[text=blue] at (13.1, 3.0) {$194$};
\node[text=blue] at (13.1, 4.5) {$55$};
\node[text=blue] at (14.6, 0.0) {$52$};
\node[text=blue] at (14.6, 1.5) {$197$};
\node[text=blue] at (14.6, 3.0) {$64$};
\node[text=blue] at (14.6, 4.5) {$201$};
\node[text=blue] at (16.1, 0.0) {$207$};
\node[text=blue] at (16.1, 1.5) {$58$};
\node[text=blue] at (16.1, 3.0) {$195$};
\node[text=blue] at (16.1, 4.5) {$54$};
\node[text=blue] at (11.6, 5.8) {$80$};
\node[text=blue] at (11.6, 7.3) {$185$};
\node[text=blue] at (11.6, 8.8) {$68$};
\node[text=blue] at (11.6, 10.3) {$181$};
\node[text=blue] at (13.1, 5.8) {$179$};
\node[text=blue] at (13.1, 7.3) {$70$};
\node[text=blue] at (13.1, 8.8) {$191$};
\node[text=blue] at (13.1, 10.3) {$74$};
\node[text=blue] at (14.6, 5.8) {$77$};
\node[text=blue] at (14.6, 7.3) {$188$};
\node[text=blue] at (14.6, 8.8) {$65$};
\node[text=blue] at (14.6, 10.3) {$184$};
\node[text=blue] at (16.1, 5.8) {$178$};
\node[text=blue] at (16.1, 7.3) {$71$};
\node[text=blue] at (16.1, 8.8) {$190$};
\node[text=blue] at (16.1, 10.3) {$75$};
\node[text=blue] at (11.6, 11.6) {$241$};
\node[text=blue] at (11.6, 13.1) {$8$};
\node[text=blue] at (11.6, 14.6) {$253$};
\node[text=blue] at (11.6, 16.1) {$12$};
\node[text=blue] at (13.1, 11.6) {$14$};
\node[text=blue] at (13.1, 13.1) {$251$};
\node[text=blue] at (13.1, 14.6) {$2$};
\node[text=blue] at (13.1, 16.1) {$247$};
\node[text=blue] at (14.6, 11.6) {$244$};
\node[text=blue] at (14.6, 13.1) {$5$};
\node[text=blue] at (14.6, 14.6) {$256$};
\node[text=blue] at (14.6, 16.1) {$9$};
\node[text=blue] at (16.1, 11.6) {$15$};
\node[text=blue] at (16.1, 13.1) {$250$};
\node[text=blue] at (16.1, 14.6) {$3$};
\node[text=blue] at (16.1, 16.1) {$246$};
\node[text=blue] at (11.6, 17.4) {$144$};
\node[text=blue] at (11.6, 18.9) {$121$};
\node[text=blue] at (11.6, 20.4) {$132$};
\node[text=blue] at (11.6, 21.9) {$117$};
\node[text=blue] at (13.1, 17.4) {$115$};
\node[text=blue] at (13.1, 18.9) {$134$};
\node[text=blue] at (13.1, 20.4) {$127$};
\node[text=blue] at (13.1, 21.9) {$138$};
\node[text=blue] at (14.6, 17.4) {$141$};
\node[text=blue] at (14.6, 18.9) {$124$};
\node[text=blue] at (14.6, 20.4) {$129$};
\node[text=blue] at (14.6, 21.9) {$120$};
\node[text=blue] at (16.1, 17.4) {$114$};
\node[text=blue] at (16.1, 18.9) {$135$};
\node[text=blue] at (16.1, 20.4) {$126$};
\node[text=blue] at (16.1, 21.9) {$139$};
\node[text=blue] at (17.4, 0.0) {$240$};
\node[text=blue] at (17.4, 1.5) {$25$};
\node[text=blue] at (17.4, 3.0) {$228$};
\node[text=blue] at (17.4, 4.5) {$21$};
\node[text=blue] at (18.9, 0.0) {$19$};
\node[text=blue] at (18.9, 1.5) {$230$};
\node[text=blue] at (18.9, 3.0) {$31$};
\node[text=blue] at (18.9, 4.5) {$234$};
\node[text=blue] at (20.4, 0.0) {$237$};
\node[text=blue] at (20.4, 1.5) {$28$};
\node[text=blue] at (20.4, 3.0) {$225$};
\node[text=blue] at (20.4, 4.5) {$24$};
\node[text=blue] at (21.9, 0.0) {$18$};
\node[text=blue] at (21.9, 1.5) {$231$};
\node[text=blue] at (21.9, 3.0) {$30$};
\node[text=blue] at (21.9, 4.5) {$235$};
\node[text=blue] at (17.4, 5.8) {$145$};
\node[text=blue] at (17.4, 7.3) {$104$};
\node[text=blue] at (17.4, 8.8) {$157$};
\node[text=blue] at (17.4, 10.3) {$108$};
\node[text=blue] at (18.9, 5.8) {$110$};
\node[text=blue] at (18.9, 7.3) {$155$};
\node[text=blue] at (18.9, 8.8) {$98$};
\node[text=blue] at (18.9, 10.3) {$151$};
\node[text=blue] at (20.4, 5.8) {$148$};
\node[text=blue] at (20.4, 7.3) {$101$};
\node[text=blue] at (20.4, 8.8) {$160$};
\node[text=blue] at (20.4, 10.3) {$105$};
\node[text=blue] at (21.9, 5.8) {$111$};
\node[text=blue] at (21.9, 7.3) {$154$};
\node[text=blue] at (21.9, 8.8) {$99$};
\node[text=blue] at (21.9, 10.3) {$150$};
\node[text=blue] at (17.4, 11.6) {$48$};
\node[text=blue] at (17.4, 13.1) {$217$};
\node[text=blue] at (17.4, 14.6) {$36$};
\node[text=blue] at (17.4, 16.1) {$213$};
\node[text=blue] at (18.9, 11.6) {$211$};
\node[text=blue] at (18.9, 13.1) {$38$};
\node[text=blue] at (18.9, 14.6) {$223$};
\node[text=blue] at (18.9, 16.1) {$42$};
\node[text=blue] at (20.4, 11.6) {$45$};
\node[text=blue] at (20.4, 13.1) {$220$};
\node[text=blue] at (20.4, 14.6) {$33$};
\node[text=blue] at (20.4, 16.1) {$216$};
\node[text=blue] at (21.9, 11.6) {$210$};
\node[text=blue] at (21.9, 13.1) {$39$};
\node[text=blue] at (21.9, 14.6) {$222$};
\node[text=blue] at (21.9, 16.1) {$43$};
\node[text=blue] at (17.4, 17.4) {$81$};
\node[text=blue] at (17.4, 18.9) {$168$};
\node[text=blue] at (17.4, 20.4) {$93$};
\node[text=blue] at (17.4, 21.9) {$172$};
\node[text=blue] at (18.9, 17.4) {$174$};
\node[text=blue] at (18.9, 18.9) {$91$};
\node[text=blue] at (18.9, 20.4) {$162$};
\node[text=blue] at (18.9, 21.9) {$87$};
\node[text=blue] at (20.4, 17.4) {$84$};
\node[text=blue] at (20.4, 18.9) {$165$};
\node[text=blue] at (20.4, 20.4) {$96$};
\node[text=blue] at (20.4, 21.9) {$169$};
\node[text=blue] at (21.9, 17.4) {$175$};
\node[text=blue] at (21.9, 18.9) {$90$};
\node[text=blue] at (21.9, 20.4) {$163$};
\node[text=blue] at (21.9, 21.9) {$86$};
\end{tikzpicture}
\vspace{.15in}
\caption{A $4\times4\times4\times4$ most-perfect magic tesseract with entries $1,\dots,256$, drawn as a $4\times4$ array of $4\times4$ arrays.
Every row, column, pillar, and file, every $2\times2$ subsquare (in all six pairs of coordinate directions), and every great diagonal and {great pandiagonal} sums to the magic constant $514$.  In addition, every pair of entries two apart on \vspace{.125in}
any great pandiagonal sums to $257$---half the magic constant. 
\linebreak This Magic Tesseract thus has all the analogous properties in four  dimensions that the Khajuraho Magic Square has in two dimensions and
Figure~\ref{fig:cube64} has in three dimensions; 
\linebreak
it represents  the first-known construction of a most-perfect magic hypercube in a dimension larger than 3.
By Theorem~\ref{thm:final2}, such most-perfect hypercubes of order 4  exist  in every dimension $m$, and for each $m$ they lie in a single orbit for a faithful action of $W(B_{2m})$.} 
\label{fig:tesseract256}
\end{figure}

As with Theorems~\ref{thm:main} and~\ref{thm:associative}, the proofs underlying Theorems~\ref{thm:final1} and~\ref{thm:final2} enable explicit constructions in all dimensions. 
We have constructed representative models of a most-perfect magic cube and a most-perfect tesseract explicitly in  Figures~\ref{fig:cube64} and~\ref{fig:tesseract256}, respectively. It is addictive to check that, in each of these examples, every orthogonal line, great diagonal, great pandiagonal, and $2\times 2$ subsquare sums to the same magic sum, 130 and 514, respectively.

\vspace{.025in}

\subsection*{\normalsize Methods and Organization}

\vspace{.05in}
In Section~\ref{sec:coordinates}, for a field $F$ of characteristic not equal to 2, we construct a particularly convenient basis for the vector space $V_n(F)$ of $2\times \cdots \times 2$ {\it face-magic} hypercubical matrices over $F$, i.e.,   $2\times \cdots \times 2$ matrices such that every $2\times2$ square face has the same sum. The basis has a particularly nice interpretation in terms of Fourier analysis on $(\mathbb Z/2\mathbb Z)^n$. 

To prove Theorem~\ref{thm:main}, we then wish to classify those matrices in $V_n(\mathbb Q)$ that have the consecutive integer entries $1,\ldots,2^n$.  Due to our choice of basis, which was designed for the purpose, this amounts to two rather elegant problems in additive number theory---one for even $n$ and one for odd $n$.  We describe these two problems precisely in Section~\ref{sec:mfhproofs}.

In Section~\ref{sec:lemmas}, we solve these two problems; the two problems loosely relate to the classical problem famously posed by Moser~\cite{Moser1957} and considered by Selfridge and Straus~\cite{SelfridgeStraus1958} regarding when a set of real numbers can be re-constructed from its subset sums.  But these two problems both have surprises in the form of exceptional/sporadic solutions.  However, the exceptional solutions arise for odd values of~$n$ in the problem relevant for $n$~even, and for even values of~$n$ in the problem relevant for $n$~odd.  Considering the relevant solutions for each of the two parities of $n$ yields proofs of Theorems~\ref{thm:main}--\ref{thm:folded2} and, with additional constraints imposed, Theorem~\ref{thm:associative}. 

In Section~\ref{sec:higher}, we then turn to the classical problem of understanding most-perfect squares, cubes, and hypercubes.   We introduce the exceptional graph isomorphism $C_4^m\cong Q_{2m}$ to reduce problems relating to compactness, pandiagonality, completeness, and most-perfectness to ones relating to corresponding properties of magic-faced hypercubes.  This perspective also helps to make all the symmetries of the situation conceptually evident, and leads to proofs of Theorems~\ref{thm:1920}--\ref{thm:implications}. 

Finally, in Section~\ref{sec:additional}, we describe some additional properties of interest satisfied by the most-perfect magic objects constructed in Theorem~\ref{thm:final2} and illustrated in Figures~\ref{fig:cube64} and \ref{fig:tesseract256}. Specifically, we study pan-$r$-diagonality for other values of $r$---a theme of frequent interest in the literature.  We end by discussing future directions.

\section{A convenient basis for face-magic hypercubes over a field of characteristic not 2}
\label{sec:coordinates}

Let $F$ be a field of characteristic not equal to 2, and let $V_n(F)$ denote the $F$-vector space of $2\times2\times\cdots\times2$ hypercubical matrices over~$F$. Let $U_n(F)\subset V_n(F)$ denote the subspace of matrices $A\in V_n(F)$ that are {\it face-magic}, i.e., those $A\in V_n(F)$ for which the entries of every $2\times 2$ square face of $A$ has the same magic sum $C=C(A)$.

We denote an element $A\in V_n(F)$ as 
\begin{equation}
A=(a_{i_1i_2\cdots i_n})_{i_k\in\{0,1\}},
\end{equation}
where each $a_{i_1i_2\cdots i_n}\in F$. We also denote the entries $a_{i_1i_2\cdots i_n}$ of $A$ more simply as~$a_S$, where \begin{equation}S=\{j: i_j=1\}\subseteq\{1,\ldots,n\}.\end{equation} That is, we write 
\begin{equation}A=(a_S)_{S\subseteq\{1,\dots,n\}}.
\end{equation}
Then $A$ is \emph{face-magic} if every $2\times2$ face of the hypercube~$A$ has the same magic sum $C=C(A)$; that is, for every $S\subseteq\{1,\dots,n\}$ and every pair of distinct indices $i,j\notin S$, we have 
\begin{equation}\label{asteq}
a_S + a_{S\cup\{i\}} + a_{S\cup\{j\}} + a_{S\cup\{i,j\}} \;=\; C. 
\end{equation}

Let $U_n^0(F)\subset U_n(F)$ denote the subspace of $A\in U_n(F)$ such that $a_\emptyset=0$. By subtracting $a_\emptyset$ from every entry of a matrix $A\in U_n(F)$, we can map any $A\in U_n(F)$ to $ U^0_n(F)$, yielding a projection $\pi:U_n(F)\to U_n^0(F)$.  

Now let us fix a matrix $A=(a_S)_{S\subseteq\{1,\dots,n\}}\in U_n^0(F)$ with magic constant $C$. For each $i\in\{1,\ldots,n\}$, let us formally write
\begin{equation}\label{firsteq}
a_{\{i\}} = b_0+\cdots+\widehat{b_i}+\cdots+b_n,
\end{equation}
where we use the symbol $\;\widehat{}\;$ to denote omission, and also write 
\begin{equation}\label{secondeq}
C=2(b_0+\cdots+b_n).
\end{equation} This gives a system of $n+1$ linear equations in $b_0, \ldots, b_n$ in terms of the entries $a_{\{i\}}$ of $A$, for $i=1,\ldots,n$, and~the magic constant $C$. The determinant of this linear change of variable from $\{a_{\{1\}},\ldots,a_{\{n\}},C\}$ to $\{b_0,b_1,\ldots,b_n\}$
is
\begin{equation}
  \left|\,
  \begin{matrix}
    1 & 0 & 1 & \!\cdots\! & 1 & 1 \\
    1 & 1 & \,0\, & \!\cdots\! & 1 & 1 \\
    \vdots & \vdots & \vdots & \!\ddots\! & \vdots & \vdots \\
    1 & 1 & 1 & \!\cdots\! & \,0\, & 1 \\
    1 & 1 & 1 & \!\cdots\! & 1 & \,0\, \\
    \,2\, & \,2\, & 2 & \!\cdots\! & 2 & 2
  \end{matrix}\,\right|
  \;=\;
  \left|\,
  \begin{matrix}
    0 & \!-1\! & 1 & \!\cdots\! & 0 & 0 \\
    0 & 0 & \!-1\! & \!\cdots\! & 0 & 0 \\
    \vdots & \vdots & \vdots & \!\ddots\! & & \vdots \\
    0 & 0 & 0 & \!\cdots\! & \!-1\! & 1 \\
    0 & 0 & 0 & \!\cdots\! & 0 & \!-1\! \\
    2 & 2 & 2 & \!\cdots\! & 2 & 2
  \end{matrix}\,\right|
  \,=\, 2.
\end{equation}
Since 2 is invertible in $F$, the 
system of linear equations (\ref{firsteq}) and (\ref{secondeq}) 
has a unique solution for $b_0,b_1,\ldots,b_n$ in terms of $a_{\{1\}},\ldots,a_{\{n\}}$ and $C$. We fix the values of $b_0,b_1,\ldots,b_n$ so that (\ref{firsteq}) and (\ref{secondeq})  are true. 

As the sum of every $2\times2$ face of $A$ containing $a_\emptyset$ equals $C$, we have for any $1\leq i<j\leq n$ that 
$$a_\emptyset+a_{\{i\}}+a_{\{j\}}+a_{\{i,j\}}=C;$$ hence 
\begin{eqnarray*}
a_{\{i,j\}} &=& 2(b_0+\cdots+b_n)) - (b_0+\cdots+\widehat{b_i}+\cdots+b_n) - (b_0+\cdots+\widehat{b_j}+\cdots+b_n) \\ &=& b_i+b_j. 
\end{eqnarray*}
Similarly, for any $1\leq i<j<k\leq n$, we have 
\[
  a_{\{i\}} + a_{\{i,j\}} + a_{\{i,k\}} + a_{\{i,j,k\}} = C,
\]
which implies that 
\begin{eqnarray*}
  a_{\{i,j,k\}} &=& 2(b_0 + \cdots + b_n) - (b_0+\cdots+\widehat{b_i}+\cdots+b_n) - (b_i+b_j) - (b_i + b_k) 
  \\ &=& b_0 + \cdots + \widehat{b_i} + \cdots + \widehat{b_j} +
\cdots + \widehat{b_k} + \cdots + b_n.
\end{eqnarray*}
In general, we compute that 
\begin{equation}\label{aSformula}
  a_S \:=\:
  \begin{cases}
    \;\displaystyle\sum_{i \in S} b_i & \text{if } |S| \text{ is even}, \\[16.5pt]
    \;\displaystyle\sum_{i \notin S} b_i & \text{if } |S| \text{ is odd},
  \end{cases}
\end{equation}
where in the second sum $i$ ranges over all values not in $S$ including $i=0$. 
Indeed, for any $S \subseteq \{1, \ldots, n\}$ and
$j,k \notin S$, the identities
\begin{eqnarray*}
  a_S + a_{S \cup \{j\}} + a_{S \cup \{k\}} + a_{S \cup \{j,k\}}
  &\!\!=\!\!& \begin{cases}
  \,\displaystyle{\sum_{i \in S} b_i + \!\sum_{i \notin S \cup \{j\}} b_i + \!\sum_{i \notin S \cup \{k\}} b_i +
    \!\sum_{i \in S \cup \{j,k\}} b_i} &  \text{if } |S| \text{ is even}
    \\[18pt]
     \,\displaystyle{\sum_{i \notin S} b_i + \!\sum_{i \in S \cup \{j\}} b_i + \!\sum_{i \in S \cup \{k\}} b_i +
    \!\sum_{i \notin S \cup \{j,k\}} b_i} &  \text{if } |S| \text{ is odd}
    \end{cases} \\[10pt]
  &\!\!=\!\!& 2 \sum_i b_i
\end{eqnarray*}
imply the formula $(\ref{aSformula})$ by induction on $|S|$.

Conversely, if $a_S$ satisfies $(\ref{aSformula})$ for all $S$, then  checking directly that 
\begin{equation}
a_S + a_{S \cup \{j\}} + a_{S \cup \{k\}} + a_{S \cup \{j,k\}} = 2 \sum_i b_i,
\end{equation} 
for all $S \subseteq \{1, \ldots, n\}$ and
$j,k \notin S,$ shows that $A$ is face-magic.

We have obtained the following theorem.

\begin{theorem}
\label{thm:coordinates}
An $n$-dimensional $2\times2\times\cdots\times2$ face-magic hypercube $A=(a_S)_{S\subseteq\{1,\dots,n\}}\in U_n(F)$ can be expressed uniquely in the form $A=A(a_\emptyset,b_0,b_1,\dots,b_n)$, where
\begin{equation}\label{mainasformula}
a_S \;=\;
\begin{cases}
a_\emptyset + \displaystyle\sum_{i\in S} b_i, & |S| \text{ even}\\[16pt]
a_\emptyset + \displaystyle\sum_{i\notin S} b_i, & |S| \text{ odd}
\end{cases}
\end{equation}
for a unique vector of arbitrary constants $(a_\emptyset,b_0,b_1,\dots,b_n)\in F^{n+2}$. Hence the $F$-vector space $U_n(F)$ of face-magic $n$-dimensional hypercubes over any field $F$ of characteristic not equal to~$2$ is $(n+2)$-dimensional, with basis $A(1,0,\dots,0),\,\ldots\,,A(0,\ldots,0,1)$.
\end{theorem}

\medskip
When $n$ is even, this set of coordinates $(a_\emptyset,b_0,\dots,b_n)$ on $U_n(F)$ immediately shows that every $A\in U_n(F)$ exhibits the face-magic folded-hypercube structure of Theorem~\ref{thm:folded}.  
Indeed, for a set $S\subseteq\{1,\ldots,n\}$, let us write 
$\overline{S}:=\{1,\ldots,n\}-S$.  
Then the sum of the four entries of a {\it new} $2\times 2$ square (i.e., 4-cycle) on $A$, when now viewed as a folded hypercube, takes the form
\begin{equation}\label{foldcheck}
a_S+a_{S\cup\{j\}}+a_{\overline{S}}+a_{\overline{S}-\{j\}}
\end{equation}
for some $S\subseteq\{1,\ldots,n\}$ and $j\notin S$. 
Since $n$ is even, the sets $S$, $S\cup\{j\}$, $\overline{S}$, $\overline{S}-\{j\}$ have alternating parity of cardinality;  hence, using  (\ref{aSformula}) on $\pi(A)$, we see that the coefficients of $b_0$, $b_j$, and indeed all other $b_i$'s in (\ref{foldcheck}) are all equal to~2.  Hence $A$ is a {\it face-magic folded hypercube}, i.e., all 4-cycles on $A$, even when viewed as a folded hypercube, have the same sum $4a_\emptyset+2(b_0+\cdots+b_n)$.

When $n$ is odd, our choice of coordinates similarly exhibits a
face-magic hierarchical-folded-hypercube structure, provided that some
$b_j$ with $j\in\{1,\ldots,n\}$ is equal to $0$. Indeed, consider the two
slices of $A$ where the $j$th coordinate is fixed to be $0$ or $1$,
respectively; that~is, the entries of the two slices take the form $a_S$ and $a_{S\cup\{j\}}$ for $S\subseteq\{1,\ldots,n\}-\{j\}$. Since $n-1$ is even, $S$ and $\overline{S}-\{j\}$ always have the same
parity of cardinality; hence, for 
$k\notin S$ and $k\neq j$, using
(\ref{aSformula}) we see,  by checking the coefficient of $b_0$, $b_j$, $b_k$, and all other $b_i$'s, that
\begin{equation}
  a_S+a_{\overline{S}-\{j\}}+a_{\overline{S}-\{j,k\}}+a_{S\cup\{k\}}
  =4a_\emptyset+2(b_0+\cdots+b_n),
\end{equation}
\begin{equation}
  a_{S\cup\{j\}}+a_{\overline{S}}+a_{\overline{S}-\{k\}}+a_{S\cup\{j,k\}}
  =4a_\emptyset+2(b_0+\cdots+b_n).
\end{equation}
Therefore, each slice, augmented to a folded hypercube $FQ_{n-1}$, has the same
constant $4$-cycle sum $4a_\emptyset+2(b_0+\cdots+b_n)$.
It remains to check the $4$-cycles crossing between the two slices, of
the form $S,\ \overline{S}-\{j\},\ \overline{S},\ S\cup\{j\}$. Using
(\ref{aSformula}) once more, we see that 
\begin{equation}
  a_S+a_{\overline{S}-\{j\}}+a_{\overline{S}}+a_{S\cup\{j\}}
  =4a_\emptyset+2(b_0+\cdots+\widehat{b_j}+\cdots+b_n)
  +\begin{cases}0,&|S|\text{ even},\\4b_j,&|S|\text{ odd}.\end{cases}
\end{equation}
This equals the common slice sum $4a_\emptyset+2(b_0+\cdots+b_n)$ for
both parities of $|S|$ at once if and only if $b_j=0$.  In that case, $A$
becomes a face-magic hierarchical folded hypercube of dimension $n$.

\pagebreak 
Thus we have proven the following theorem.

\begin{theorem}\label{thm:extra}
If $n$ is even, then any face-magic hypercube of dimension $n$ is automatically a face-magic folded hypercube of dimension $n$.

If $n$ is odd, then a face-magic hypercube $A(a_\emptyset,b_0,\ldots,b_n)$ satisfying $b_j = 0$ for some $j\in\{1,\ldots,n\}$ is automatically a face-magic hierarchical $($in the $j$-th direction$)$ folded hypercube of dimension $n$.
\end{theorem}

In the next section, we apply these results to understand  magic-faced hypercubes over the integers.

\begin{remark}
\label{rem:fourier}
{\em The proofs of Theorems~\ref{thm:coordinates} and \ref{thm:extra} have the following interpretation in terms of Fourier analysis on $(\mathbb Z/2\mathbb Z)^n$, which is indeed how we first discovered these theorems. 

Let $F$ be an abelian group in which the doubling map $x\mapsto 2x=x+x$ is bijective. (A~field of characteristic not $2$ is an example, though no multiplicative structure on $F$ is needed or used anywhere below.) For a function $A:(\mathbb Z/2\mathbb Z)^n\to F$, define its Fourier transform
\begin{equation}
\hat A(y) \;=\; \sum_{x\in(\mathbb Z/2\mathbb Z)^n} (-1)^{\langle x,y\rangle} A(x), \qquad y\in(\mathbb Z/2\mathbb Z)^n,
\end{equation}
where $\langle x,y\rangle=\sum_i x_iy_i\in\mathbb Z/2\mathbb Z$.

Then $A$ is face-magic if and only if $\hat A(y)=0$ for every $y$ having weight in the set $\{1,2, \ldots, n-2\}$. In other words, $\hat A$ is supported only on weights $0$, $n-1$, and $n$. This immediately proves the dimension count of Theorem~\ref{thm:coordinates}: there are $1+n+1=n+2$ such $y\in (\mathbb Z/2\mathbb Z)^n$, implying that $\dim_F U_n(F)=n+2$.  

Write $\mathbf 0:=(0,\ldots,0)\in(\mathbb Z/2\mathbb Z)^n$.  Let $e_0:=(1,\ldots,1)\in(\mathbb Z/2\mathbb Z)^n$. For $k=1,\ldots,n$, let $e_k\in(\mathbb Z/2\mathbb Z)^n$ be the vector with a $0$ in position $k$ and $1$'s in every other position. Thus $\mathbf 0, e_0,\ldots,e_n$ are exactly the $n+2$ values of $y\in (\mathbb Z/2\mathbb Z)^n$ having weight 0, $n$, or $n-1$.  
Since doubling is bijective on $F$, so is $x\mapsto 2^nx$. Hence the usual Fourier inversion formula holds: 
\begin{equation}\label{finversion}
A(x) \;=\; \frac{1}{2^n}\sum_{y\in(\mathbb Z/2)^n} (-1)^{\langle x,y\rangle}\,\hat A(y).
\end{equation}
As $\hat A$ vanishes outside $\{\mathbf 0,e_0,\ldots,e_n\}$, this collapses to
\begin{equation}\label{inft}
A(x) \;=\; \frac{1}{2^n}\left[\hat A(\mathbf 0) + \sum_{k=0}^n \hat A(e_k)\,(-1)^{\langle x,e_k\rangle}\right].
\end{equation}
That is, any face-magic hypercube of dimension $n$ over $F$ is spanned by the all-1's hypercube (corresponding to $y=\mathbf 0$) together with the hypercubes given by $(-1)^{\langle x,e_k\rangle}$ (corresponding to $y=e_k$)  for~$k=0,1,\ldots,n$. 
This gives a natural set of $n+2$ independent coordinates on face-magic hypercubes of dimension $n$ with entries in $F$, namely, $\hat A(\mathbf 0), \hat A(e_0),\ldots,\hat A(e_n)$; these can take arbitrary values in $F$ to specify a face-magic hypercube of dimension~$n$ with entries in $F$, via formula $(\ref{inft})$. 

Theorem~\ref{thm:coordinates} used a different but related set of coordinates (designed to translate the classification of magic-faced hypercubes with consecutive integer entries into an elegant number theory problem, as studied in Section~\ref{sec:mfhproofs}). The relation is as follows. Let  
\begin{equation}a_\emptyset:=A(\mathbf 0)
\end{equation}
and 
\begin{equation}
b_k:=-\frac{\hat A(e_k)}{2^{n-1}}
\end{equation}
for $k=0,\ldots,n$. 
 Evaluating (\ref{inft}) at $x=\mathbf 0$ gives 
\begin{equation}
A(\mathbf 0)=2^{-n}\bigl[\hat A(\mathbf 0)+\sum_k\hat A(e_k)\bigr].
\end{equation}  
Therefore, for each $S\subseteq\{1,\ldots,n\}$ with indicator vector $x_S$, we see that formula~(\ref{inft}), in terms of the coordinates $a_\emptyset, b_0,\ldots,b_n$, becomes 
\begin{equation}\label{mainasformula2}
a_S \;=\; a_\emptyset + \frac12\sum_{k=0}^n b_k\Bigl(1-(-1)^{\langle x_S,\,e_k\rangle}\Bigr). 
\end{equation}
This is equivalent to formula (\ref{mainasformula}), yielding   
Theorem~\ref{thm:coordinates}. 

Theorem~\ref{thm:extra} also then follows directly from formula (\ref{inft}). Indeed, the all-1's hypercube is clearly a face-magic $FQ_n$; and, for any $y\in\{e_0,\ldots,e_n\}$, the hypercube $x\mapsto(-1)^{\langle x,y\rangle}$ is also a face-magic $FQ_n$ when $n$ is even.
Since being a face-magic $FQ_n$ is a linear condition, any $A=(a_S)_{S\subset\{1,\ldots,n\}}$ defined by (\ref{inft}) or (\ref{mainasformula2}) is also always then a face-magic folded hypercube of dimension~$n$.  
This perspective also isolates associativity cleanly: 
the all-1's hypercube and the hypercube $x\mapsto(-1)^{\langle x,e_k\rangle}$ for any $k\geq 1$ is associative, while $x\mapsto(-1)^{\langle x,e_0\rangle}$ is not associative. Since associativity is again a linear condition, the hypercube $A$ in (\ref{mainasformula2}) or (\ref{mainasformula}) is associative if and only if $b_0=0$.

Similarly, if $n$ is odd, then the hypercube $x\mapsto(-1)^{\langle x,e_k\rangle}$ for $k\ge1$ is a face-magic $H_jFQ_n$ (in the sense of Theorem~\ref{thm:extra}) for every direction $j\ne k$, but fails to be so for the direction $j=k$ itself.  By contrast, the all-1's hypercube and the hypercube $x\mapsto(-1)^{\langle x,e_0\rangle}$ are face-magic $H_jFQ_n$'s for every direction $j$. Thus the hypothesis $b_j=0$ in Theorem~\ref{thm:extra} is precisely the condition required to ensure that the hypercube $A$ in (\ref{mainasformula2}) or (\ref{mainasformula}) is a face-magic $H_jFQ_n$, yielding Theorem~\ref{thm:extra} for odd~$n$.}
\end{remark}

\section{Magic-faced hypercubes over the integers: Proofs of Theorems~\ref{thm:main}--\ref{thm:associative}}
\label{sec:mfhproofs}

We now apply the results of Section~\ref{sec:coordinates} to classify face-magic  hypercubes $A$ of dimension~$n$ having the consecutive integer entries $1,\ldots,2^n$.  It will be convenient in the sequel to subtract~1 from all entries, and thereby obtain a face-magic hypercube $A$ having entries $0,\ldots,2^n-1$. 
Furthermore, by a suitable symmetry of the hypercube graph underlying $A=(a_S)_{S\subseteq\{1,\ldots,n\}}$, we may assume that $a_\emptyset$ is the smallest entry 0 of the hypercube $A$.  Thus $A$ is an element of~$U_n^0(\mathbb Q)$. 

We now apply Theorem~\ref{thm:coordinates} to classify such face-magic matrices $A\in U_n^0(\mathbb Q)$ with $a_\emptyset=0$ and entries $0,\ldots,2^{n}-1$. The choice of coordinates in the previous section was  carefully designed to make this a tractable number theory problem---and the number theory problem turns out to have a very elegant answer. 

Indeed, with the choice of coordinates given in Theorem~\ref{thm:coordinates}, we have $a_\emptyset=0$, and so  only the coordinates $b_0,\ldots,b_n$ remain.  We break naturally into two cases: $n$ even and $n$ odd. 

\subsection{The parity of $n$ is even}

If $n$ is even, then the set of entries of $A(0,b_0,\ldots,b_n)$ (as given by Theorem~\ref{thm:coordinates}) exhibits a beautiful structure: it consists of all sums of even-sized subsets of $\{b_0,\ldots,b_n\}$.  The question then reduces to the following: for what values of $b_0,\ldots,b_n$ are the $2^n$ sums of even-sized subsets of $\{b_0,\ldots,b_n\}$ equal to the consecutive integers $0,\ldots,2^n-1$ in some order?

This question is answered by the following combinatorial lemma.

\begin{lemma}\label{lemma1}
Let $n\geq 0$ be an even integer. Suppose $B=\{b_0,\ldots,b_n\}$ is a set of $n+1$ real numbers such that the $2^n$ sums of even-sized subsets of $B$ yield $0,\ldots,2^{n}-1$ in some~order.  Then $\{b_0,\ldots,b _n\}=\{0\} \cup \{1,2,4,\ldots,2^{n-1}\}$.
\end{lemma}
\noindent
It is not hard to see, by the uniqueness of binary expansions, that this choice of $\{b_0,\ldots,b _n\}$ does indeed give the desired even-sized subset sums.  However, somewhat surprisingly, the fact that there are no other solutions requires the condition that $n$ is even!  Specifically, if $n=3$, then one may take $\{b_0,b_1,b_2,b_3\}=\{-\frac12,\frac32,\frac52,\frac72\}$; the even-sized subset sums of this set then give the eight integers 0,1,2,3,4,5,6,7.  However, no such counterexamples exist once~$n\geq 4$. We will prove this statement, and in particular, Lemma~\ref{lemma1}, in Section~\ref{sec:lemmas}.

Taking Lemma~\ref{lemma1} to be true for the time being, we conclude that for $A=A(0,b_0,\ldots,b_n)$ to be a face-magic hypercube of dimension $n$ with $n$ even and entries $0,\ldots,2^n-1$, we must have that $b_0,\ldots,b_n$ is a permutation of $0,1,2,4,\ldots,2^{n-1}$.  Furthermore, the resulting action of a permutation in $S_{n+1}$ on $\{b_0,\ldots,b_n\}$ (and thereby on the entries of $A$ via the formulas in Theorem~\ref{thm:coordinates}) corresponds exactly to the action of the automorphism of the folded hypercube graph induced by permuting the $n+1$ edges emanating from $a_\emptyset$ when $A$ is viewed as a face-magic folded hypercube of dimension~$n$. Indeed, pairs of entries adjacent in the folded hypercube $A(0,b_0,\ldots,b_n)$ remain adjacent after the permutation in $S_{n+1}$ is applied to $b_0,\ldots,b_n$. 

There is a group of symmetries of the folded hypercube of dimension $n$ that is complementary to this $S_{n+1}$, namely, the $(\mathbb Z/2)^n$ that flips the slices of the hypercube in each direction and thereby moves the 0 entry to the $2^n$ other entries.  The group $C_2^n\rtimes S_{n+1}$ represents the full symmetry group of the folded hypercube of dimension $n$.  We have proven that there is exactly one magic-faced folded hypercube of dimension $n$, up to symmetries of the folded hypercube, as desired.  Theorems~\ref{thm:main}--\ref{thm:folded2} thus follow in the case of even $n$, assuming the truth of Lemma~\ref{lemma1} which we will establish in the next section. 

\vspace{.1in}
\noindent
{\bf Associativity.} Finally, we determine which elements of the $C_2^n\rtimes S_{n+1}$-orbit of magic-faced folded hypercubes of dimension $n$ are associative when $n$ is even. We have proven that there is a magic-faced folded hypercube $A_n$ of dimension $n$ which is unique up to automorphisms of the underlying graph $FQ_n$.  The proof shows that pairs of numbers in $A_n$ that add up to $2^n+1$ are always adjacent to each other in this folded hypercube.  Moreover, if these adjacencies (edges) in~$A_n$ are removed, what remains is an associative magic-faced hypercube of dimension $n$.  By Theorem~\ref{thm:folded2}(b), this magic-faced hypercube is therefore the unique associative magic-faced hypercube of dimension $n$ up to symmetries of the hypercube. 

Alternatively, if $n$ is even, then using (\ref{aSformula}) we see that for any magic-faced hypercube $A=(a_S)$, we have 
\begin{equation}
a_S+a_{\overline{S}} \;=\; \begin{cases} 2^n-1-b_0, & |S| \text{ even},\\ 2^n-1+b_0, & |S| \text{ odd}. \end{cases} 
\end{equation}
for all $S\subset\{1,\ldots,n\}$. For $a_S+a_{\overline{S}}$ to be a constant across $|S|$ even and $|S|$ odd, we require $b_0=0$. Thus $A(0,b_1,\ldots,b_n)$, where $\{b_1,\ldots,b_n\}=\{1,2,4,\ldots,2^{n-1}\}$, gives a magic-faced hypercube with entries $0,\ldots,2^n-1$, and all such hypercubes form a single orbit for the action of the group $\Aut(Q_n)=C_2^n\rtimes S_{n} \subset C_2^n\rtimes S_{n+1} = \Aut(FQ_n)$, where $S_n\subset S_{n+1}$ acts on $\{b_1,\ldots,b_n\}$ while fixing $b_0$. This completes the proof of Theorem~\ref{thm:associative} for even $n$, again assuming the truth of Lemma~\ref{lemma1} which we will prove in the next section.

\subsection{The parity of $n$ is odd}

If $n$ is odd, then the set of entries of $A(0,b_0,\ldots,b_n)$ exhibits an
analogous structure: writing~$C$ for the set of sums of even-sized
subsets of $\{b_1,\ldots,b_n\}$, the entries of $A$ are exactly
$$C\,\cup\,(b_0+C).$$ The question then reduces to the following: for what
values of $b_0,\ldots,b_n$ is $C\,\cup\,(b_0+C)$ equal to the consecutive
integers $0,\ldots,2^n-1$ in some order?

This question is answered by the following combinatorial lemma.
\begin{lemma}\label{lemma2}
Let $n\ge3$ be an odd integer, and let $b_0,\ldots,b_n$ be real numbers
such that $C\,\cup\,(b_0+C)=\{0,\ldots,2^n-1\}$, where $C$ is the set of
sums of even-sized subsets of $\{b_1,\ldots,b_n\}$. Then
$\{b_0,\ldots,b_n\}=\{0\}\cup\{1,2,4,\ldots,2^{n-1}\}$, with $b_0\ne0$.
\end{lemma}
\noindent
It is not hard to see that any such choice of $\{b_0,\ldots,b_n\}$, with
$b_0$ taken to be any one of the $n$ nonzero elements, yields
$C\,\cup\,(b_0+C)=\{0,\ldots,2^n-1\}$. (Taking $b_0=0$ results in 
$C=b_0+C$, so their union has only $2^{n-1}$ elements rather than $2^n$.)
However, in this case, the fact that there are no other solutions requires the condition that $n$ is odd!  For example, if $n=4$, then one may take $b_0=8$ and $\{b_1,b_2,b_3,b_4\}=\{-\frac12,\frac32,\frac52,\frac72\}$; then we have $C\,\cup\,(b_0+C)=\{0,\ldots,15\}$. However, no such counterexamples exist once~$n\geq 5$. 
We will prove this statement, and in particular, Lemma~\ref{lemma2}, in Section~\ref{sec:lemmas}.

Taking Lemma~\ref{lemma2} to be true for the time being, we conclude
that for $A=A(0,b_0,\ldots,b_n)$ to be a face-magic hypercube of
dimension $n$ with $n$ odd and entries $0,\ldots,2^n-1$, we must have that
$b_0,\ldots,b_n$ is a permutation of $0,1,2,4,\ldots,2^{n-1}$ with
$b_j=0$ for some $j\in\{1,\ldots,n\}$---which by Theorem~\ref{thm:extra} is the hypothesis under
which $A$ is a face-magic hierarchical folded hypercube, splitting into
the two folded hypercubes of dimension $n-1$ determined by the $j$th
coordinate.

The resulting action of $S_n$ on $\{b_0,\ldots,\widehat{b_j},\ldots,b_n\}$ (and
thereby on the entries of~$A$, with $j$~fixed) corresponds to
the automorphisms of this hierarchical folded hypercube that fix the
distinguished direction~$j$. Together with the complementary~$C_2^n$ of
coordinate flips, this accounts for the group~$C_2^n\rtimes S_n$ of
symmetries for each fixed~$j$. Letting~$j$ range over~$\{1,\ldots,n\}$ 
then accounts for the remaining freedom regarding which coordinate~$j$ is
distinguished, recovering the full symmetry structure described in
Theorem~\ref{thm:main}. We have proven that there is exactly one magic-faced
hierarchical folded hypercube of dimension $n$, up to these symmetries,
as desired. Theorems~\ref{thm:main}--\ref{thm:folded2} now follow in the case of odd
$n$, assuming the truth of Lemma~\ref{lemma2} which we will establish in the next section.

\pagebreak 
\vspace{.1in}
\noindent 
{\bf Associativity.} Finally, for odd values of $n$, we determine which elements of the $C_2^n\rtimes S_{n}$-orbit of magic-faced hierarchical folded hypercubes of dimension $n$ are associative. We have proven that there is a magic-faced hierarchical folded hypercube $A_n$ of dimension $n$ which is unique up to automorphisms of the underlying graph $HFQ_n$.  Our proof again implies that pairs of numbers in the magic-faced hierarchical folded hypercube $A_n$ that sum to $2^n+1$ are always adjacent to each other.  However, in this case, removing these adjacencies does not result in a hypercube but a disconnected graph, namely, two disjoint copies of $FQ_{n-1}$.  It follows that there is no associative magic-faced hypercube of dimension $n$.

Alternatively, if $n$ is odd, then using (\ref{aSformula}), we see that for any magic-faced hypercube $A=(a_S)$, we have 
$|a_S-a_{\overline{S}}|=|b_0|$ (and thus not $|a_S+a_{\overline{S}}|$) is a constant for all $S\subset\{1,\ldots,n\}$. Therefore, there does not exist any associative magic-faced hypercube of dimension $n$ when $n$ is odd.   This completes the proof of   Theorem~\ref{thm:associative} for odd $n$, again assuming the truth of Lemma~\ref{lemma2} which we will prove in the next section. 

\section{Proofs of Lemmas}
\label{sec:lemmas}

In this section, we prove Lemmas~\ref{lemma1} and \ref{lemma2}. 
This will thereby complete the proofs of Theorems~\ref{thm:main}--\ref{thm:folded2}.

\subsection{An Auxiliary Lemma}

We begin with an auxiliary classical fact about subset sums that  Lemmas~\ref{lemma1} and \ref{lemma2} both rest on. It states that if the subset sums of an $n$-element set $T$ of real numbers consist of the consecutive integers $0,1,\ldots,2^{n}-1$ in some order, then $T$ must be the set $\{1,2,4,\dots,2^{n-1}\}$ of the first $n$ powers of~2. 

\begin{lemma}
\label{lem:auxiliary}
Let $n\ge1$ and let $b_1,\dots,b_n\in\mathbb{R}$. If the $2^n$ sums of all possible subsets of
$\{b_1,\ldots,b_n\}$ are equal to $0,1,\ldots,2^{n}-1$ in some order, then 
\[
\{b_1,\dots,b_n\}=\{1,2,4,\dots,2^{n-1}\}.
\]
\end{lemma}

\begin{proof}
Let $\{b_1,\dots,b_n\}$ be a set of real numbers such that the $2^n$ subset sums are equal to $0,1,\ldots,2^{n}-1$ in some order. Without loss of generality, we assume that $b_1\leq b_2\leq \cdots \leq b_n$.  
We prove that we must have $b_i=2^{i-1}$ by induction on $i$.

For a subset $T\subseteq\{b_1,\dots,b_n\}$, let $\sigma(T)$ denote the sum of the elements of $T$. Then $\sigma(\emptyset)=0$. Since $b_1$ is itself one of the $2^n$ subset sums, it lies in $\{0,\ldots,2^n-1\}$, so $b_1\ge0$; and since $b_1$ is the smallest of the $b_i$'s, all $b_i\ge0$. The smallest value of $\sigma(T)$ for $T\neq \emptyset$ is thus $\sigma(\{b_1\})=b_1$. Thus we must have $b_1=1$, proving the desired statement for $i=1$. 

Now assume that the values of $b_1,b_2,\ldots,b_i$ are equal to $1,2,4,\ldots,2^{i-1}$, respectively.  Then the subset sums of $\{b_1,\ldots,b_i\}=\{1,2,4,\ldots,2^{i-1}\}$ are equal to $0,\ldots,2^i-1$, by the uniqueness of $i$-digit binary expansions of the nonnegative integers up to $2^{i}-1$.  The smallest value of $\sigma(T)$ for $T\not\subseteq \{b_1,\ldots,b_i\}$ is $\sigma(\{b_{i+1}\})=b_{i+1}$.  It follows that $b_{i+1}$ must be the smallest number not yet attained, namely, $2^i$, proving the desired assertion for $i+1$.  This result now follows by induction. 
\end{proof}

This proof will be our model for the next two propositions, which naturally generalize Lemmas~\ref{lemma1} and \ref{lemma2}.

\pagebreak 

\subsection{The First Proposition}

Proposition~\ref{lem:even} aims to generalize Lemma~\ref{lem:auxiliary} from all subsets to all {\it even-sized} subsets: if the even-sized subset sums of an $(n+1)$-element set $T$ of real numbers consist of the consecutive integers $0,1,\ldots,2^{n}-1$ in some order, then what does $T$ look like?

\begin{proposition}
\label{lem:even}
Let $n\geq 0$ be an integer. Suppose $B=\{b_0,\ldots,b_n\}$ is a set of $n+1$ real numbers such that the $2^n$ sums of even-sized subsets of $B$ are equal to $0,\ldots,2^{n}-1$ in some~order.   Then: 
\begin{itemize}
\item[{\em (a)}] If $n=0$, this holds for every $b_0\in\mathbb{R}$. 
\item[{\em (b)}] If $n=1$, this holds if and only if $\{b_0,b_1\}=\{a,1-a\}$ for some real number $a$. 
\item[{\em (c)}] If $n=3$, this holds if and only if 
\begin{equation}
\:\:\;\;\;\;\;\;\;\{b_0,b_1,b_2,b_3\} = \{0\} \cup \{1,2,4\} \text{ or } \bigl\{-\tfrac12,\tfrac32,\tfrac52,\tfrac72\bigr\}.
\end{equation}
\item[{\em (d)}] If $n=2$ or $n\ge4$, this holds if and only if
\begin{equation}
 \{b_0,\ldots,b _n\}=\{0\} \cup \{1,2,4,\ldots,2^{n-1}\}.
\end{equation}
\end{itemize}
\end{proposition}

\begin{proof}
If $n=0$, the only even-sized subset of $\{b_0\}$ is $\emptyset$, whose sum is $0$ regardless of $b_0$; hence the single required sum of $0$ is attained for every $b_0\in\mathbb{R}$.

If $n=1$, the even-sized subsets of $\{b_0,b_1\}$ are $\emptyset$ and $\{b_0,b_1\}$, with sums $0$ and $b_0+b_1$. We need $\{0,b_0+b_1\}=\{0,1\}$, implying that $b_0+b_1=1$, i.e.\ $\{b_0,b_1\}=\{a,1-a\}$ for any $a\in \mathbb R$.

Suppose now that $n\ge2$. Without loss of generality, we assume that $b_0\le b_1\le\cdots\le b_n$. For $T\subseteq\{0,\ldots,n\}$ write $\sigma(T)=\sum_{i\in T}b_i$. Then $\sigma(\emptyset)=0$.

\vspace{.1in}
\noindent\textbf{Case 1: $b_0=0$.} Every even-sized $T$ either omits $b_0$, in which case $\sigma(T)$ is an even-sized subset sum of $\{b_1,\ldots,b_n\}$, or contains $b_0$, in which case $T=\{b_0\}\cup T'$ with $T'\subseteq\{b_1,\ldots,b_n\}$ odd-sized, and $\sigma(T)=b_0+\sigma(T')=\sigma(T')$. So the $2^n$ even-sized subset sums of $\{b_0,\ldots,b_n\}$ are exactly the $2^n$ subset sums---of every size---of $\{b_1,\ldots,b_n\}$. By Lemma~\ref{lem:auxiliary}, applied verbatim to $b_1\le\cdots\le b_n$, this holds if and only if $\{b_1,\ldots,b_n\}=\{1,2,4,\ldots,2^{n-1}\}$; together with $b_0=0$ this gives $\{b_0,\ldots,b_n\}=\{0\}\cup\{1,2,4,\ldots,2^{n-1}\}$.

\vspace{.1in}
\noindent\textbf{Case 2: $b_0\ne0$.} Since only $\emptyset$ attains $0$, every nonempty even-sized subset $T$ has sum $\sigma(T)\ge1$.  Furthermore, if $|T|\ge4$, then splitting $T$ into two disjoint nonempty even-sized pieces shows that $\sigma(T)$ is a sum of two \emph{distinct} positive integers, implying that $\sigma(T)\ge3$. 

Hence the smallest nonzero value, $1$, and the next, $2$, are each attained as $\sigma(T)$ where $T$ is a pair.  The smallest pair sums are obtained by $T=\{b_0,b_1\}$ and $T=\{b_0,b_2\}$.  Thus $\sigma(\{b_0,b_1\}=b_0+b_1=1$ and $b_0+b_2=2$. 

Now a subset $T$ of size $\ge4$ attaining $\sigma(T)=3$ would have to split into disjoint pairs $T_1,T_2$ with $\sigma(T_1)=1$ and $\sigma(T_2)=2$, and thus $T_1=\{b_0,b_1\}$ and $T_2=\{b_0,b_2\}$; this is not possible as these $T_1,T_2$ overlap.  Thus $3$ too must be attained by a pair; hence either $\sigma(\{b_1,b_2\})=b_1+b_2=3$ or, if $n\ge3$, $\sigma(\{b_0,b_3\})=b_0+b_3=3$, as $b_1+b_2$ and $b_0+b_3$ are the two smallest pair sums not yet used. 

However, adding $b_0+b_1=1$ and $b_0+b_2=2$ gives $2b_0+(b_1+b_2)=3$, so $b_1+b_2=3$ would force $b_0=0$, contrary to assumption. Hence $n\ge3$ and $b_0+b_3=3$. The equations $2b_0+(b_1+b_2)=3$ and $b_0+b_3=3$ then imply
\begin{equation}
b_3=b_0+b_1+b_2.
\end{equation}
If $n\ge4$, then $\sigma(\{b_3,b_4\})=b_3+b_4=b_0+b_1+b_2+b_4=\sigma(\{b_0,b_1,b_2,b_4\})$, yielding two distinct even-sized subsets with the same sum---contradicting the distinctness of all $2^n$ even-sized subset sums. Therefore, $n=3$ remains the only possibility.

If $n=3$, the only remaining even-sized subset that needs to be determined is $\{b_1,b_2\}$, whose sum must equal the only remaining value: $b_1+b_2=4$. Together with $2b_0+(b_1+b_2)=3$, this gives $b_0=-\tfrac12$, and hence $b_1=\tfrac32$, $b_2=\tfrac52$, and $b_3=3-b_0=\tfrac72$. We check directly that the even-sized subset sums of $\{-\tfrac12,\tfrac32,\tfrac52,\tfrac72\}$  indeed give $0,\ldots,7$.
\end{proof}

Proposition~\ref{lem:even} shows that if the even-sized subset sums of a set $U$ of size $n+1\geq 2$ are equal to $0,\ldots,2^{n}-1$ in some order, then either $U=\{0,1,2,4,\ldots,2^{n-1}\}$ (the standard solution) or $U$ is some exceptional set that exists only when $n$ is odd (indeed, either $n=1$ or $n=3$). Lemma~\ref{lemma1} follows, and hence so do Theorems~\ref{thm:main}--\ref{thm:folded2} in the even cases. 

\subsection{The Second Proposition}

Proposition~\ref{lem:odd} aims to generalize Lemma~\ref{lem:auxiliary} from all subsets to all {\it even-sized} subsets and their translates by a fixed integer: if the even-sized subset sums of an $n$-element set $T$ of real numbers, together with their translates by a fixed integer $b_0$ consist of the consecutive integers $0,1,\ldots,2^{n}-1$ in some order, then what do $T$ and $b_0$ look like?

\begin{proposition}
\label{lem:odd}
Let $n\ge1$ be an integer, and let $b_0,\ldots,b_n$ be real numbers
such that $C\,\cup\,(b_0+C)=\{0,\ldots,2^n-1\}$, where $C$ is the set of
sums of even-sized subsets of $\{b_1,\ldots,b_n\}$. Then: 
\begin{itemize}
\item[{\em (a)}] If $n=1$, this holds if and only if $b_0=1$ while $b_1\in \mathbb R$ is arbitrary. 
\item[{\em (b)}] If $n=2$, this holds if and only if $b_0\in\{1,2\}$ and $\{b_1,b_2\}=\{a,3-b_0-a\}$ for some real number $a$. 
\item[{\em (c)}] If $n=4$, this holds if and only if 
$b_0\in\{1,2,4,8\}$ and  
\begin{eqnarray}
\{b_1,b_2,b_3,b_4\}&=&\{0\}\cup (\{1,2,4,8\}\setminus\{b_0\}) 
\:\!\text{ or }\:\! \\[7pt]
\{b_1,b_2,b_3,b_4\}&=&\left\{\tfrac{15-b_0}{2},\tfrac{15-b_0}{2}\!-\!b_1,\tfrac{15-b_0}{2}\!-\!b_2,\tfrac{15-b_0}{2}\!-\!b_3\right\}.
\end{eqnarray}
\item[{\em (d)}] If $n=3$ or $n\geq 5$, this holds if and only if 
\begin{equation}
\{b_0,\ldots,b_n\}=\{0\}\cup\{1,2,4,\ldots,2^{n-1}\}, \text{ with } b_0\ne0.
\end{equation}
\end{itemize}
\end{proposition}

\begin{proof}
If $n=1$, the only even-sized subset of $\{b_1\}$ is $\emptyset$, so
$C=\{0\}$ regardless of $b_1$. We require $\{0,b_0\}=\{0,1\}$, implying
that $b_0=1$, while $b_1\in\mathbb R$ is unconstrained.

If $n=2$, the even-sized subsets of $\{b_1,b_2\}$ are $\emptyset$ and
$\{b_1,b_2\}$, so $C=\{0,b_1+b_2\}$. Write $s=b_1+b_2$. We require
$\{0,s,b_0,b_0+s\}=\{0,1,2,3\}$, i.e., $\{s,b_0,b_0+s\}=\{1,2,3\}$.
Since $b_0\ge1$ (shown below), the choice $s=3$ forces $b_0+s\ge4$,
a contradiction.  The choice $s=1$ gives $\{b_0,b_0+1\}=\{2,3\}$, i.e.,  $b_0=2$.  Finally, the choice $s=2$ implies $\{b_0,b_0+2\}=\{1,3\}$, i.e., $b_0=1$. Hence $b_0\in\{1,2\}$
and $b_1+b_2=3-b_0$.  We conclude that $\{b_1,b_2\}=\{a,3-b_0-a\}$ where $a=b_1\in\mathbb R$ is arbitrary.

Suppose now that $n\ge3$. Without loss of generality, we assume that
$b_1\le\cdots\le b_n$. For $T\subseteq\{1,\ldots,n\}$, write
$\sigma(T)=\sum_{i\in T}b_i$, so that $C=\{\sigma(T):|T|\text{ even}\}$.
Then $\sigma(\emptyset)=0$.

\vspace{.1in}
\noindent\textbf{Case 1: $b_1=0$.} Every even-sized $T\subseteq\{1,\ldots,n\}$
either omits index $1$, in which case $\sigma(T)$ is an even-sized
subset sum of $\{b_2,\ldots,b_n\}$, or contains it, in which case
$T=\{1\}\cup T'$ with $T'\subseteq\{2,\ldots,n\}$ odd-sized, and
$\sigma(T)=b_1+\sigma(T')=\sigma(T')$. So $C$ is exactly the set of
\emph{all} subset sums of $\{b_2,\ldots,b_n\}$, and hence $C\cup(b_0+C)$
is exactly the set of all subset sums of the $n$ numbers
$b_0,b_2,\ldots,b_n$. By Lemma~\ref{lem:auxiliary}, $C\cup(b_0+C)$ is equal to $\{0,\ldots,2^n-1\}$ if and only if
$\{b_0,b_2,\ldots,b_n\}=\{1,2,\ldots,2^{n-1}\}$.  Together with $b_1=0$, 
this gives $\{b_0,\ldots,b_n\}=\{0\}\cup\{1,2,\ldots,2^{n-1}\}$ with
$b_0\ne0$.  This is the generic solution, valid for every $n\ge3$.

\vspace{.1in}
\noindent\textbf{Case 2: $b_1\ne0$.} Since $T=\emptyset$ attains $\sigma(T)=0$, 
every nonempty $T$ satisfies $\sigma(T)\ge1$. In particular, 
$b_0=b_0+\sigma(\emptyset)$ is also an attained
value, and so is a positive integer. Moreover, $b_0$ is a power of $2$:
partitioning $\{0,\ldots,2^n-1\}$ into consecutive blocks of length
$b_0$, no element of $C$ is negative, so $[0,b_0)\subseteq C$; hence
$[b_0,2b_0)\subseteq b_0+C$; hence $[2b_0,3b_0)\subseteq C$; and so~on. 
For this to split $\{0,\ldots,2^n-1\}$ into two equal
halves, $2^n/b_0$ must be an even integer, i.e.,~$b_0\mid2^{n-1}$.

Each index $i\in\{1,\ldots,n\}$ lies in
exactly $2^{n-2}$ of the $2^{n-1}$ even-sized subsets of
$\{1,\ldots,n\}$, so $\sum_{c\in C}c=2^{n-2}\sum_{i=1}^n b_i$. Since $C$
and $b_0+C$ partition $\{0,\ldots,2^n-1\}$, adding $\sum_{c\in C}c$ and
$\sum_{c\in C}(b_0+c)$ gives $\sum_{k=0}^{2^n-1}k=2^{n-1}(2^n-1)$.
We conclude that 
\begin{equation}
  \sum_{i=1}^n b_i=2^n-1-b_0.
\end{equation}

Let the two smallest positive elements of $C$ be denoted by $c_1$ and 
$c_2$, with $c_1<c_2$. If
$b_0=1$, then $C=\{0,2,\ldots,2^n-2\}$, so $c_1=2,c_2=4$. If $b_0=2$, then
$[0,2)\subseteq C$ but $[2,4)\subseteq b_0+C$, and $[4,6)\subseteq C$,
so $c_1=1,c_2=4$. If $b_0\ge4$, then $[0,b_0)\subseteq C$ already contains
$1,2$, so $c_1=1,c_2=2$. In each case, $c_1+c_2$ is the
\emph{third} smallest positive element of $C$.

This is exactly the configuration of Case 2 in the proof of
Lemma~\ref{lem:even}, with $(b_1,c_1,c_2)$ in place of $(b_0,1,2)$.
Splitting any $|T|\ge4$ into two disjoint nonempty pieces shows
that $\sigma(T)$ is a sum of two distinct elements of $C\setminus\{0\}$,
so $\sigma(T)\ge c_1+c_2$. Thus $c_1,c_2$ are each attained by a pair,
and taking the smallest pair sums gives
\begin{equation}
  b_1+b_2=c_1,\qquad b_1+b_3=c_2.
\end{equation}
A subset of size $\ge4$ attaining $c_1+c_2$ would split into disjoint
pairs attaining $c_1$ and $c_2$, forcing them to be $\{b_1,b_2\}$ and
$\{b_1,b_3\}$---impossible, as these overlap. So $c_1+c_2$ too must be
attained by a pair; hence either $\sigma(\{b_2,b_3\})=b_2+b_3=c_1+c_2$
or, if $n\ge4$, $\sigma(\{b_1,b_4\})=b_1+b_4=c_1+c_2$, as $b_2+b_3$ and
$b_1+b_4$ are the two smallest pair sums not yet used.

Adding $b_1+b_2=c_1$ and $b_1+b_3=c_2$ gives $2b_1+(b_2+b_3)=c_1+c_2$,
so $b_2+b_3=c_1+c_2$ would force $b_1=0$, contrary to assumption. Hence
$n\ge4$ (if $n=3$, no such $b_4$ exists, and both branches fail, a
contradiction) and $b_1+b_4=c_1+c_2$, giving
\begin{equation}\label{bc}
  b_2=c_1-b_1,\qquad b_3=c_2-b_1,\qquad b_4=c_1+c_2-b_1.
\end{equation}
If $n\ge5$, then $\sigma(\{b_4,b_5\})=b_4+b_5=b_1+b_2+b_3+b_5=
\sigma(\{b_1,b_2,b_3,b_5\})$ (using $b_4=b_1+b_2+b_3$), contradicting the distinctness of
all $2^n$ sums. Therefore $n=4$ remains the only possibility.

For $n=4$, substituting the three expressions in (\ref{bc}) into the sum
identity $b_1+b_2+b_3+b_4=15-b_0$ yields
\begin{equation}
  2b_1-2(c_1+c_2)=-(15-b_0).
\end{equation}
Writing $s:=\tfrac{15-b_0}{2}$ then gives
\begin{equation}
  b_1=c_1+c_2-s,\qquad b_2=s-c_2,\qquad b_3=s-c_1,\qquad b_4=s.
\end{equation}
By Case 1, the set $\{0\}\cup(\{1,2,4,8\}\setminus\{b_0\})$ is a
solution for this choice of $b_0$. Writing its nonzero elements as $c_1,c_2,c_3$ (so $c_1+c_2+c_3=(1+2+4+8)-b_0=2s$), the four values above
are exactly $s-c_3,\,s-c_2,\,s-c_1,\,s-0$, corresponding to the four stated exceptional solutions as $b_0$ ranges over $\{1,2,4,8\}$. They are worked out explicitly in Table~\ref{tab:anomaly}; we check directly that the even-sized subset sums of $\{b_1,b_2,b_3,b_4\}$, together with their translates by $b_0$,   indeed give $0,\ldots,15$.
\end{proof}

\begin{table}[t]
\centering
\begin{tabular}{c|c}
$b_0$ & $\{b_1,b_2,b_3,b_4\}$ \\[2pt]
\hline
$1$ & $\{-1,3,5,7\}^{\phantom{|^{|^?}}}\!\!\!$ \\[4pt]
$2$ & $\bigl\{-\tfrac32,\tfrac52,\tfrac{11}2,\tfrac{13}2\bigr\}$ \\[5pt]
$4$ & $\bigl\{-\tfrac52,\tfrac72,\tfrac92,\tfrac{11}2\bigr\}$ \\[5pt]
$8$ & $\bigl\{-\tfrac12,\tfrac32,\tfrac52,\tfrac72\bigr\}$
\end{tabular}
\caption{The exceptional solutions when $n=4$.
In~each~case, one checks that the even-sized-subset sums of $\{b_1,b_2,b_3,b_4\}$, together with their translations by~$+b_0$, yield the consecutive integers $0,1,2,3,\ldots,16$ in some order.}
\label{tab:anomaly}
\end{table}

Proposition~\ref{lem:odd} shows that if the even-sized subset sums of a set $U$ of size $n\geq 2$, together with its translates by a fixed number $b_0$ are equal to $0,\ldots,2^{n}-1$ in some order, then either $\{b_0\}\cup U=\{0,1,2,4,\ldots,2^{n-1}\}$ with $b_0\neq 0$ (the standard solution) or $\{b_0\}\cup U$ is some exceptional set that exists only when $n$ is even (indeed, either $n=2$ or $n=4$). Lemma~\ref{lemma2} follows, and hence so do Theorems~\ref{thm:main}--\ref{thm:folded2} in the odd cases.

\begin{remark}
\label{rem:char0}
{\em Although the proofs of Lemmas~\ref{lemma1} and~\ref{lemma2} given above make use of the ordering of $\mathbb R$ (e.g., in arguing that $b_1$ must be the smallest element, or that certain sums must be the smallest attainable value), the resulting classification statements hold over {any} field of characteristic $0$, not merely over $\mathbb R$ or $\mathbb Q$. Indeed, by the determinant-$2$ computation underlying Theorem~\ref{thm:coordinates}, each $b_i$ is an integer linear combination of the entries $a_S$, divided by $2$. Since the entries in question are the integers $0,\ldots,2^n-1$, every $b_i$ necessarily lies in $\frac12\mathbb Z\subset\mathbb Q$, regardless of which field of characteristic $0$ it is a priori permitted to lie in. Therefore, the classifications in  Lemmas~\ref{lemma1} and~\ref{lemma2}, stated over $\mathbb R$, are classifications over any field of characteristic $0$.

The same reasoning shows that, for each fixed $n$, only finitely many candidate $(n+1)$-tuples $(b_0,\ldots,b_n)\in\bigl(\tfrac12\mathbb Z\bigr)^{n+1}$ need ever be checked, each determined from a choice of which subset sum matches which target value by the same integer linear algebra.  Thus Lemmas~\ref{lemma1} and~\ref{lemma2} continue to hold, by the identical arguments, over any field of characteristic $p$ for sufficiently large $p$.  However, there can be exceptions for ``small'' $p$; for instance, we will see in Example~\ref{ex:f17} that there are exceptional solutions in characteristic $17$ when $n=4$.}
\end{remark}

\section{The existence, classification, and enumeration of \linebreak compact and most-perfect magic hypercubes in each dimension: Proofs of Theorems~\ref{thm:final1}--\ref{thm:implications}}
\label{sec:higher}

\subsection{Proof of Theorem~\ref{thm:final1}}

We first consider  compact magic hypercubes of order 4 and dimension $m$, with entries $1,\ldots,4^m=2^{2m}$.  Let $M$ be such a compact $4\times 4\times \cdots 4$ magic hypercube of dimension $m$.  Due to the graph isomorphism $C_4^m\cong K_2^{2m}\cong Q_{2m}$, we can map $M$ to an element~$A\in V_{2m}(\mathbb Q)$ having entries $1,\ldots,2^{2m}$.  
Using the fact that  orthogonal lines and $2\times 2$ subsquares of $M$
together correspond to the square $2\times 2$ faces of the corresponding $2m$-dimensional hypercube $A$, we see that in fact $A\in U_{2m}(\mathbb Q)$. Theorem~\ref{thm:main} now implies the existence, classification, and enumeration of compact magic hypercubes of every dimension, and in particular, Theorem~\ref{thm:final1}. 

\subsection{Proof of Theorem~\ref{thm:final2}}

We now add on the condition of completeness.  Completeness (which implies pandiagonality) of $M$ is equivalent to the associativity of $A$.  Theorem~\ref{thm:associative} then yields the existence, classification, and enumeration of  most-perfect magic hypercubes of order $4$, and in particular Theorem~\ref{thm:final2}.

\subsection{Proof of Theorem~\ref{thm:implications}}

To classify those compact magic hypercubes of order 4 and dimension $m$ that are pandiagonal, it remains to determine, among the assignments of the $n+1$ values $\{0,2^0,\dots,2^{n-1}\}$ to $b_0,\dots,b_n$, exactly which ones also make every pandiagonal of $A(0,b_0,\ldots,b_n)$ sum to the same constant. 

Let $M$ be a compact magic hypercube of order 4 and dimension $m$, and let $A$ be the corresponding magic-faced hypercube of dimension $2m$ via the exceptional graph isomorphism $C_4^m\cong Q_{2m}$. Then a great pandiagonal of $M$ corresponds to four entries of $A=(a_{i_1\cdots i_{2m}})_{\i_j\in\{0,1\}}$ of the form $a_w,a_{w'},a_{w''},a_{w'''}$, where: $w=w_1\ldots w_{2m}$ is any string of $0's$ and $1's$ of length~$2m$; $w'=w'_1\ldots w'_{2m}$ is obtained by flipping $m$ bits of $w$; $w''$ is obtained by flipping the other $m$~bits of $w$; and $w'''$ is obtained by flipping all the bits of $w$.  

We now use (\ref{aSformula}) to compute the sum $a_w+a_{w'}+a_{w''}+a_{w'''}$.  This is where the parity of $m$ enters.  Write $a_w=a_S$, $a_{w'}=a_{S'}$, $a_{w''}=a_{S''}$, and $a_{w'''}=a_{S'''}$, where again $S=\{j:w_j=1\}$, $S'=\{j:w'_j=1\}$, etc.  

If $m$ is even, then $S,S',S'',S'''$ all have the same parity, and so (\ref{aSformula}) gives 
\begin{equation}
a_S+a_{S'}+a_{S''}+a_{S'''}=
\begin{cases} \displaystyle \,2\sum_{i=1}^{2m}b_i, & |S| \text{ even},\\[16.5pt] \displaystyle \,2\sum_{i=1}^{2m}b_i + 4b_0, & |S| \text{ odd}. \end{cases} 
\end{equation}
We obtain the magic sum in all cases if and only if $b_0=0$, i.e., if $A$ is associative and $M$ is complete.  Thus if $m$ is even, then compact and pandiagonal implies compact, pandiagonal, and complete, i.e., most-perfect.

If $m$ is odd, then $S,S',S''',S''$ have alternating parities, so (\ref{aSformula}) gives 
\begin{equation}
a_S+a_{S'}+a_{S''}+a_{S'''}\,=\,  2\sum_{i=0}^{2m}b_i 
\end{equation}
regardless of the parity of $S$.  Thus, if $m$ is odd, then compact automatically implies compact and pandiagonal.

We have proven Theorem~\ref{thm:implications}.

\section{Additional Properties and Future Directions}
\label{sec:additional}

The existence of most-perfect magic hypercubes in every dimension raises a number of questions, regarding their properties and future directions, which we hope will be considered in future work.  In this section, we describe three such questions that naturally arise, relating to: 1) {\it pan-$r$-agonality} of most-perfect magic hypercubes for values of $r$ different from the dimension; 2) {\it non-normal} most-perfect hypercubes, i.e., those not having consecutive integer entries; and 3) extensions to orders higher than~$4$.  We discuss each of these in turn. 

\subsection{Pan-$r$-agonality properties of most-perfect magic hypercubes}

The tesseract of Figure~\ref{fig:tesseract256}---and, more generally, the  most-perfect hypercubes guaranteed to exist by Theorem~\ref{thm:final2}---in fact automatically satisfy far more than Theorem~\ref{thm:final2} asserts. A natural question is whether only the entries along great diagonals of a most-perfect hypercube of dimension $m$ have the same magic sum, or whether the entries along some of the diagonals of smaller $r$-dimensional sub-hypercubes (known as $r$-agonals) might also have the same sum. The notion of $r$-agonals, and pan-$r$-agonals, also summing to the same magic constant has been a common theme in the literature on magic squares, cubes, and hypercubes.  In this language, orthogonal lines are $1$-agonals, while great diagonals are $m$-agonals; but we may consider $r$-agonals with $1<r<m$. 

In the case of Figure~\ref{fig:tesseract256}, not only does each great pandiagonal (all four coordinates moving together) sum to the same magic constant, but so does every ``pan-$3$-agonal'' (any three of the four coordinates moving together, the fourth held fixed at a fixed value)!  However, the ``pan-$2$-agonals'' (two coordinates moving, two fixed) do {\it not} all sum to the magic constant and, indeed, no most-perfect magic tesseract exists where the pan-$2$-agonals also all yield the same magic sum. 

The following theorem explains this phenomenon, for every dimension $m$. 
In particular, it explains the difference in behavior for odd and even $m$, and when the 
associativity of the magic-faced hypercube plays a role. 
 
\begin{theorem}
\label{thm:ragonal}
Let $n=2m$, and let $A$ be a magic-faced hypercube of dimension $n$.  Let $M$ be the corresponding compact magic hypercube of order $4$ obtained via the isomorphism $C_4^{\,m}\cong Q_n$ with coordinates indexed by $(\mathbb Z/4\mathbb Z)^m$. For $1\le r\le m$, define a  \emph{pan-$r$-agonal} to be a line obtained by fixing $m-r$ of the $m$ coordinates at given values and letting the remaining $r$ coordinates move  by a vector of the form $(\pm 1,\ldots,\pm 1)$ in $(\mathbb Z/4\mathbb Z)^r$. Then:
\begin{enumerate}
\item[{\em (a)}] If $r$ is odd, then every pan-$r$-agonal of $M$ sums to the magic constant $2^{n+1}+2$, regardless of whether $A$ is associative.
\item[{\em (b)}] If $r=m$ is even, then every pan-$r$-agonal of $M$ sums to the magic constant $2^{n+1}+2$ if and only if $A$ is associative.
\item[{\em (c)}] If $r<m$ and $r$ is even, then the pan-$r$-agonals of $M$ do not all sum to a common value, regardless of $A$. 
\end{enumerate}
\end{theorem}

\begin{proof}
The proof generalizes that of Theorem~\ref{thm:implications}. A pan-$r$-agonal of $M$ corresponds to four entries of $M=(a_{i_1\cdots i_{2m}})_{i_j\in\{0,1\}}$ of the form $a_w,a_{w'},a_{w''},a_{w'''}$, where $w=w_1\ldots w_{2m}$ is any string of $0$'s and $1$'s of length $2m$; $w'$ is obtained by flipping $r$ of the bits belonging to the $r$ moving coordinates of $w$; $w''$ is obtained by flipping the other $r$ bits amongst these moving coordinates; and $w'''$ is obtained by flipping all $2r$ of these bits. The remaining $2(m-r)$ bits, belonging to the $m-r$ fixed coordinates, remain unchanged throughout.

We again use (\ref{aSformula}) to compute $a_w+a_{w'}+a_{w''}+a_{w'''}$, writing $a_w=a_S,a_{w'}=a_{S'},a_{w''}=a_{S''},a_{w'''}=a_{S'''}$ as before, where now $S,S',S'',S'''$ differ from one another only within the $2r$-element set $T\subseteq\{1,\ldots,2m\}$ of flipped bits. This is where the parity of $r$ enters. 

If $r$ is odd, then $S,S',S''',S''$ have alternating parities, exactly as in the proof of Theorem~\ref{thm:implications}, so (\ref{aSformula}) gives 
\begin{equation}
a_S+a_{S'}+a_{S''}+a_{S'''}=\,2\sum_{i=0}^{2m}b_i
\end{equation}
regardless of the parity of $S$, and regardless of which $r$ of the $m$ coordinates are moving or of the values at which the remaining $m-r$ coordinates are fixed. This proves (a).

If $r$ is even, then $S,S',S'',S'''$ all have the same parity, and so (\ref{aSformula}) gives 
\begin{equation}
a_S+a_{S'}+a_{S''}+a_{S'''}=
\begin{cases} \displaystyle \,2\sum_{i\in T}b_i, & |S| \text{ even},\\[16.5pt] \displaystyle \,2\sum_{i\in T}b_i + 4b_0, & |S| \text{ odd}. \end{cases}
\end{equation}
If $r=m$, then $T=\{1,\ldots,2m\}$, and this is exactly the formula appearing in the proof of Theorem~\ref{thm:implications}: we obtain the magic sum in both cases if and only if $b_0=0$, i.e., if and only if $A$ is associative. This proves (b).

If instead $r<m$, then $T$ is a proper subset of $\{1,\ldots,2m\}$, depending on which $r$ of the $m$ coordinates are chosen to move; as $T$ ranges over the different $2r$-element subsets arising from different such choices, the value $\sum_{i\in T}b_i$ varies. Thus the pan-$r$-agonals of $M$ do not all sum to a common value, regardless of the choice of $A$. This proves (c).
\end{proof}

In particular, the case $r=m$ of Theorem~\ref{thm:ragonal} recovers Theorem~\ref{thm:implications}: if $m$ is odd, then compactness of a magic hypercube of order 4 and dimension $m$ also guarantees its pandiagonality by Theorem~\ref{thm:ragonal}(a); and if $m$ is even, then associativity is a necessary and sufficient condition for pandiagonality by Theorem~\ref{thm:ragonal}(b). 

\subsection{Non-normal magic faced hypercubes}

What can the entries of a non-normal magic-faced hypercube of dimension $n$ be?  What can the entries of a most-perfect magic hypercube of order 4 and dimension $m$ be? 

The answer is contained in the following theorem.

\begin{theorem}
\label{thm:entries}
Let $R$ be a ring in which $2$ is invertible. Let $n\ge 0$ be an integer. A set~$E$ of $2^n$ elements $e_1,\ldots,e_{2^n}\in R$ arises as the multiset of entries of a face-magic hypercube of dimension $n$ over $R$ precisely when:
\begin{itemize}
\item[{\em (a)}] If $n$ is even, then there exist $a,b_0,\ldots,b_n\in R$ such that $$E \:=\: a \,+\, \Bigl\{\displaystyle\sum_{i\in T}b_i \,:\, T\subseteq\{0,\ldots,n\} \textrm{\em \,\! with } |T|\textrm{\em \,\! even}\Bigr\},$$ i.e., $E$ is the translate of the set of even-sized subset sums of $\{b_0,\ldots,b_n\}$.
\item[{\em (b)}] If $n$ is odd, then there exist $a,b_0,\ldots,b_n\in R$ such that, writing $E_0$ for the set of even-sized subset sums of $\{b_1,\ldots,b_n\}$, we have $E = a + \bigl(E_0 \cup (b_0+E_0)\bigr)$.
\end{itemize}
\end{theorem}

\noindent
Theorem~\ref{thm:entries} follows from formula (\ref{aSformula}) for the entries of a magic-faced hypercube, which was proved  under the assumption that 2 is invertible.  Since the entries of a most-perfect hypercube of dimension $m$ and order 4 are the same as the corresponding face-magic hypercube of dimension $n=2m$, we also have the following corollary. 

\begin{corollary}\label{cor:entries}
Let $R$ be a ring in which $2$ is invertible. Let $n=2m\ge 0$ be an even integer. A~set~$E$ of $2^n$ elements $e_1,\ldots,e_{2^n}\in R$ arises as the multiset of entries of a most-perfect magic hypercube of dimension $m$ and order $4$ over $R$ precisely when there exist $a,b_0,\ldots,b_n\in R$ such that 
\begin{equation}
E \:=\: a \,+\, \Bigl\{\displaystyle\sum_{i\in T}b_i \,:\, T\subseteq\{0,\ldots,n\} \textrm{\em \,\! with } |T|\textrm{\em \,\! even}\Bigr\},
\end{equation} 
i.e., $E$ is the translate of the set of even-sized subset sums of $\{b_0,\ldots,b_n\}$.
\end{corollary}
\noindent It would be interesting to consider the $R$-module structure of $U_n(R)$ also in cases where 2 is not invertible.

Theorem~\ref{thm:entries} and Corollary~\ref{cor:entries} 
describe when face-magic hypercubes of dimension $n$ exist with given entries.  We next consider the question of uniqueness.  
The uniqueness in the case of the entries $E=\{1,\ldots,2^n\}$ over the ring $R=\mathbb R$ followed from Lemma~\ref{lemma1} (in the case of even $n$)  and Lemma~\ref{lemma2} (in the case of odd $n$); these lemmas showed that the corresponding sets $\{a,b_0,\ldots,b_n\}$ as required by Theorem~\ref{thm:entries} were essentially unique up to suitable permutations of these sets once the choice $a=1$ was made.  The proofs of these lemmas, carried out in Section~\ref{sec:lemmas}, crucially used the fact that we were working in characteristic~0 (see~Remark~\ref{rem:char0}).  

The latter hypothesis on $R$ having characteristic~0 was indeed relevant, as the following counterexample over the finite field $\mathbb F_{17}$ illustrates. 

\begin{example}
\label{ex:f17}
{\em The two sets $\{b_0,\ldots,b_4\}=\{0,1,2,4,8\}$ and $\{b_0,\ldots,b_4\}=\{0,5,6,7,14\}$ in $\mathbb F_{17}$ give rise to two distinct face-magic tesseracts over $\mathbb F_{17}$ with the identical sets of even-sized subset sums, namely, $\{0,1,2,\ldots,15\}\subset \mathbb F_{17}$. Translating both by $a=1$ gives the common set  of entries $\{1,2,\ldots,16\}$; the first hypercube recovers the dim~4 magic-faced tesseract depicted in Figure~\ref{mfh}. Therefore, moving from the ordered field $\mathbb R$ to the field $\mathbb F_{17}$ of positive characteristic yields an additional magic-faced tesseract not related to the first by any element of $\Aut(FQ_4)$. This is illustrated in Figure~\ref{fig:f17example}. 

\vspace{.15in}
\begin{figure}[htb]
\centering
\begin{subfigure}[b]{0.47\textwidth}
\centering
\begin{tikzpicture}[every path/.style={draw=red}, every node/.style={text=blue}, scale=0.56, baseline=(current bounding box.center)]
\draw (0,0) -- (10,0);
\draw (0,0) -- (0,10);
\draw (10,0) -- (10,10);
\draw (0,10) -- (10,10);
\draw (2.5,1.5) -- (8.52,1.5);
\draw (9.08,1.5) -- (9.72,1.5);
\draw (10.28,1.5) -- (12.5,1.5);
\draw (12.5,1.5) -- (12.5,11.5);
\draw (12.5,11.5) -- (2.5,11.5);
\draw (2.5,11.5) -- (2.5,10.28);
\draw (2.5,9.72) -- (2.5,8.306);
\draw (2.5,7.746) -- (2.5,2.912);
\draw (2.5,2.352) -- (2.5,1.5);
\draw (0,0) -- (2.5,1.5);
\draw (10,0) -- (12.5,1.5);
\draw (10,10) -- (12.5,11.5);
\draw (0,10) -- (2.5,11.5);
\draw (3.8,4) -- (6.8,4) -- (6.8,7) -- (3.8,7) -- cycle;
\draw (5.675,5.125) -- (6.52,5.125);
\draw (7.08,5.125) -- (8.675,5.125);
\draw (8.675,5.125) -- (8.675,8.125);
\draw (5.675,8.125) -- (7.72,8.125);
\draw (8.28,8.125) -- (8.675,8.125);
\draw (5.675,8.125) -- (5.675,7.28);
\draw (5.675,6.72) -- (5.675,5.125);
\draw (3.8,4) -- (5.675,5.125);
\draw (6.8,4) -- (8.675,5.125);
\draw (6.8,7) -- (8.675,8.125);
\draw (3.8,7) -- (5.675,8.125);
\draw (0,0) -- (3.8,4);
\draw (10,0) -- (6.8,4);
\draw (10,10) -- (6.8,7);
\draw (0,10) -- (3.8,7);
\draw (2.5,1.5) -- (4.505,3.789);
\draw (4.874,4.211) -- (5.675,5.125);
\draw (12.5,1.5) -- (10.2,3.677);
\draw (9.797,4.062) -- (8.675,5.125);
\draw (12.5,11.5) -- (10.21,9.479);
\draw (9.79,9.109) -- (8.675,8.125);
\draw (2.5,11.5) -- (3.719,10.2);
\draw (4.103,9.796) -- (5.675,8.125);
\node[fill=white,inner sep=2pt] at (0,0) {$1$};
\node[fill=white,inner sep=2pt] at (10,0) {$15$};
\node[fill=white,inner sep=2pt] at (10,10) {$4$};
\node[fill=white,inner sep=2pt] at (0,10) {$14$};
\node[fill=white,inner sep=2pt] at (2.5,1.5) {$12$};
\node[fill=white,inner sep=2pt] at (12.5,1.5) {$6$};
\node[fill=white,inner sep=2pt] at (12.5,11.5) {$9$};
\node[fill=white,inner sep=2pt] at (2.5,11.5) {$7$};
\node[fill=white,inner sep=2pt] at (3.8,4) {$8$};
\node[fill=white,inner sep=2pt] at (6.8,4) {$10$};
\node[fill=white,inner sep=2pt] at (6.8,7) {$5$};
\node[fill=white,inner sep=2pt] at (3.8,7) {$11$};
\node[fill=white,inner sep=2pt] at (5.675,5.125) {$13$};
\node[fill=white,inner sep=2pt] at (8.675,5.125) {$3$};
\node[fill=white,inner sep=2pt] at (8.675,8.125) {$16$};
\node[fill=white,inner sep=2pt] at (5.675,8.125) {$2$};
\end{tikzpicture}
\caption*{$a_\emptyset=1$, \:$(b_0,\ldots,b_4)=(0,1,2,4,8)$}
\end{subfigure}
\hfill
\begin{subfigure}[b]{0.47\textwidth}
\centering
\begin{tikzpicture}[every path/.style={draw=red}, every node/.style={text=blue}, scale=0.56, baseline=(current bounding box.center)]
\draw (0,0) -- (10,0);
\draw (0,0) -- (0,10);
\draw (10,0) -- (10,10);
\draw (0,10) -- (10,10);
\draw (2.5,1.5) -- (8.52,1.5);
\draw (9.08,1.5) -- (9.72,1.5);
\draw (10.28,1.5) -- (12.5,1.5);
\draw (12.5,1.5) -- (12.5,11.5);
\draw (12.5,11.5) -- (2.5,11.5);
\draw (2.5,11.5) -- (2.5,10.28);
\draw (2.5,9.72) -- (2.5,8.306);
\draw (2.5,7.746) -- (2.5,2.912);
\draw (2.5,2.352) -- (2.5,1.5);
\draw (0,0) -- (2.5,1.5);
\draw (10,0) -- (12.5,1.5);
\draw (10,10) -- (12.5,11.5);
\draw (0,10) -- (2.5,11.5);
\draw (3.8,4) -- (6.8,4) -- (6.8,7) -- (3.8,7) -- cycle;
\draw (5.675,5.125) -- (6.52,5.125);
\draw (7.08,5.125) -- (8.675,5.125);
\draw (8.675,5.125) -- (8.675,8.125);
\draw (5.675,8.125) -- (7.72,8.125);
\draw (8.28,8.125) -- (8.675,8.125);
\draw (5.675,8.125) -- (5.675,7.28);
\draw (5.675,6.72) -- (5.675,5.125);
\draw (3.8,4) -- (5.675,5.125);
\draw (6.8,4) -- (8.675,5.125);
\draw (6.8,7) -- (8.675,8.125);
\draw (3.8,7) -- (5.675,8.125);
\draw (0,0) -- (3.8,4);
\draw (10,0) -- (6.8,4);
\draw (10,10) -- (6.8,7);
\draw (0,10) -- (3.8,7);
\draw (2.5,1.5) -- (4.505,3.789);
\draw (4.874,4.211) -- (5.675,5.125);
\draw (12.5,1.5) -- (10.2,3.677);
\draw (9.797,4.062) -- (8.675,5.125);
\draw (12.5,11.5) -- (10.21,9.479);
\draw (9.79,9.109) -- (8.675,8.125);
\draw (2.5,11.5) -- (3.719,10.2);
\draw (4.103,9.796) -- (5.675,8.125);
\node[fill=white,inner sep=2pt] at (0,0) {$1$};
\node[fill=white,inner sep=2pt] at (10,0) {$11$};
\node[fill=white,inner sep=2pt] at (10,10) {$12$};
\node[fill=white,inner sep=2pt] at (0,10) {$10$};
\node[fill=white,inner sep=2pt] at (2.5,1.5) {$9$};
\node[fill=white,inner sep=2pt] at (12.5,1.5) {$13$};
\node[fill=white,inner sep=2pt] at (12.5,11.5) {$15$};
\node[fill=white,inner sep=2pt] at (2.5,11.5) {$14$};
\node[fill=white,inner sep=2pt] at (3.8,4) {$2$};
\node[fill=white,inner sep=2pt] at (6.8,4) {$3$};
\node[fill=white,inner sep=2pt] at (6.8,7) {$8$};
\node[fill=white,inner sep=2pt] at (3.8,7) {$4$};
\node[fill=white,inner sep=2pt] at (5.675,5.125) {$5$};
\node[fill=white,inner sep=2pt] at (8.675,5.125) {$7$};
\node[fill=white,inner sep=2pt] at (8.675,8.125) {$16$};
\node[fill=white,inner sep=2pt] at (5.675,8.125) {$6$};
\end{tikzpicture}
\caption*{$a_\emptyset=1$, \:$(b_0,\ldots,b_4)=(0,5,6,7,14)$}
\end{subfigure}
\vspace{.05in}
\caption{Two magic-faced tesseracts over $\mathbb F_{17}$, with identical entries $\{1,2,\ldots,16\}\subset \mathbb F_{17}$  and magic constant  $34\equiv 0 \pmod{17}$ on every $2\times2$ face. The first is the same dimension-$4$ example of Figure~\ref{mfh}. Despite having the same set of entries, these two tesseracts are not related by any element of $\Aut(FQ_4)$; for example, the entries 1 and 14 are adjacent on the first magic-faced hypercube but they are not adjacent on the second (even when the latter is viewed as a folded hypercube $FQ_4$).}
\label{fig:f17example}
\end{figure}

\begin{figure}[ht]
\centering
\vspace{-.125in}
\begin{subfigure}[b]{0.47\textwidth}
\centering
\begin{tikzpicture}[scale=0.7225]
\def\cs{1.25}
\draw[line width=0.9pt,xstep=\cs,ystep=\cs] (0,0) grid (4*\cs,4*\cs);
\node at (0.5*\cs,3.5*\cs) {\Large 7};
\node at (1.5*\cs,3.5*\cs) {\Large 12};
\node at (2.5*\cs,3.5*\cs) {\Large 1};
\node at (3.5*\cs,3.5*\cs) {\Large 14};
\node at (0.5*\cs,2.5*\cs) {\Large 2};
\node at (1.5*\cs,2.5*\cs) {\Large 13};
\node at (2.5*\cs,2.5*\cs) {\Large 8};
\node at (3.5*\cs,2.5*\cs) {\Large 11};
\node at (0.5*\cs,1.5*\cs) {\Large 16};
\node at (1.5*\cs,1.5*\cs) {\Large 3};
\node at (2.5*\cs,1.5*\cs) {\Large 10};
\node at (3.5*\cs,1.5*\cs) {\Large 5};
\node at (0.5*\cs,0.5*\cs) {\Large 9};
\node at (1.5*\cs,0.5*\cs) {\Large 6};
\node at (2.5*\cs,0.5*\cs) {\Large 15};
\node at (3.5*\cs,0.5*\cs) {\Large 4};
\end{tikzpicture}
\vspace{.05in}
\caption*{\!\!\!The Khajuraho Magic Square}
\end{subfigure}
\hfill
\begin{subfigure}[b]{0.47\textwidth}
\centering
\begin{tikzpicture}[scale=0.7225]
\def\cs{1.25}
\draw[line width=0.9pt,xstep=\cs,ystep=\cs] (0,0) grid (4*\cs,4*\cs);
\node at (0.5*\cs,3.5*\cs) {\Large 14};
\node at (1.5*\cs,3.5*\cs) {\Large 9};
\node at (2.5*\cs,3.5*\cs) {\Large 1};
\node at (3.5*\cs,3.5*\cs) {\Large 10};
\node at (0.5*\cs,2.5*\cs) {\Large 6};
\node at (1.5*\cs,2.5*\cs) {\Large 5};
\node at (2.5*\cs,2.5*\cs) {\Large 2};
\node at (3.5*\cs,2.5*\cs) {\Large 4};
\node at (0.5*\cs,1.5*\cs) {\Large 16};
\node at (1.5*\cs,1.5*\cs) {\Large 7};
\node at (2.5*\cs,1.5*\cs) {\Large 3};
\node at (3.5*\cs,1.5*\cs) {\Large 8};
\node at (0.5*\cs,0.5*\cs) {\Large 15};
\node at (1.5*\cs,0.5*\cs) {\Large 13};
\node at (2.5*\cs,0.5*\cs) {\Large 11};
\node at (3.5*\cs,0.5*\cs) {\Large 12};
\end{tikzpicture}
\vspace{.05in}
\caption*{Another Most-perfect Magic Square over $\mathbb F_{17}$}
\end{subfigure}
\vspace{.05in}
\caption{Two most-perfect $4\times4$ magic squares over $\mathbb F_{17}$, with identical entries $\{1,2,\ldots,16\}$ and magic sum $34\equiv0\pmod{17}$ on every row, column, $2\times2$ block, diagonal, and pan-diagonal; and sum $17\equiv 0\pmod{17}$ for any two entries whose entries are at distance 2 in each coordinate. The first is the original Khajuraho Magic Square; the second is obtained from the nonstandard magic-faced tesseract over $\mathbb F_{17}$ in Figure~\ref{fig:f17example}, via the same exceptional isomorphism~$C_4^2\cong Q_4$. \\[.1in]
Therefore, even though the Khajuraho Magic Square is the unique most-perfect $4\times4$ magic square over $\mathbb R$, up to the action of $W(B_4)$, there is an additional one in characteristic 17.}
\label{fig:khajuraho_f17}
\end{figure}

\vspace{.125in}

It follows that the Khajuraho Magic Square, while being the unique most-perfect magic square of order 4 in characteristic~0, up to the action of $\Aut(FQ_4)\cong W(B_4)$, is {\it not} the unique such square in characteristic~17---there is actually a second $W(B_4)$-orbit of such most-perfect magic squares over $\mathbb F_{17}$. This is illustrated in Figure~\ref{fig:khajuraho_f17}. }
\end{example}

\pagebreak

Example~\ref{ex:f17} is closely related to the long line of work on the ``Multiset Recovery Problem,'' which goes back to a  classical 1957 question of Moser~\cite{Moser1957}: can a multiset of numbers be uniquely recovered from the multiset of its subset sums of a fixed order? In their important work, Selfridge and Straus~\cite{SelfridgeStraus1958} gave the first significant results towards this problem in~1958, showing that a multiset of $n$ numbers is recoverable from its $2$-sums if and only if $n$ is not a power of $2$. They also constructed the first explicit failures of recovery for $3$-sums and $4$-sums. These results were extended by Gordon, Fraenkel, and Straus~\cite{GordonFraenkelStraus1962}, and later by Ewell~\cite{Ewell1968}, Boman--Bolker--O'Neil~\cite{BomanBolkerONeil1991}, and Isomurodov--Kokhas~\cite{IsomurodovKokhas2017}, who finally resolved the last outstanding case $(n,s)=(12,4)$ after nearly fifty years; see the survey of Fomin~\cite{Fomin2019} for the complete history.

The aforementioned works were primarily concerned with subset sums of a {fixed order} $s$. The variant most directly relevant to us, i.e., whether a multiset can be recovered from the multiset of \emph{all} its subset sums of every order combined---was fully resolved only recently, by Ciprietti and Glaudo~\cite{CipriettiGlaudo2025} and Glaudo and Kravitz~\cite{GlaudoKravitz2024}. They prove that an abelian group $G$ has the property that ``same total subset sums'' implies ``related by a sign flip on a zero-summing subset'' if and only if every torsion element of $G$ has order lying in an explicit set $O_{FS}$.  In particular, every torsion-free group (such as $\mathbb{Z}$, $\mathbb{Q}$, or $\mathbb{R}$) has this property automatically, while $\mathbb{Z}/17\mathbb{Z}$ is the smallest cyclic group for which it fails.  
This is what motivated us to look for a counterexample in characteristic 17 (even though here we are concerned only with even-sized subsets). 

It is interesting to ask whether the analogue for {\it even-sized} subsets has a similar structure to the theorems proven for all-sized subsets in \cite{CipriettiGlaudo2025,GlaudoKravitz2024}.  If close analogues held, that would mean that the analogues of Lemmas~\ref{lemma1} and \ref{lemma2} might hold for all sets of entries in $\mathbb R$. Either way, it would be interesting to classify collisions of the type exhibited in Example~\ref{ex:f17} over $\mathbb R$ and indeed over more general rings / abelian groups.  

We end with some of the various relevant questions that arise: 

\begin{itemize}
\item What is the structure of the $R$-module $U_n(R)$ of face magic hypercubes of dimension~$n$ over $R$ when the characteristic is even and positive?
\vspace{-.065in}
\item Given a ring $R$, for which subsets $E\subset R$ of size $2^{n}$ does there exist more than one $W(B_{n})$-orbit of face-magic hypercubes of dimension $n=2m$ (equivalently, of most-perfect magic hypercubes of dimension $m$ and order 4) with $E$ as the set of entries?   We may also ask the analogous questions for face-magic hypercubes of dimension $n$ where $n$ is odd.  
\vspace{-.065in}
\item We have seen in Example~\ref{ex:f17} that such collisions of even-sized subset sums of two sets of size $n+1$ with $n$ even can happen in  any ring of characteristic~17. Can such a collision happen in $\mathbb R$?  That is, do the analogues of Lemmas~\ref{lemma1} (and also Lemma~\ref{lemma2}) remain true for any achievable set $E\subset \mathbb R$ in place of $\{0,\ldots,2^{n}-1\}$?  
\end{itemize}

\subsection*{\normalsize Application: A most-perfect variant of D\"urer Magic Square in characteristic 17}

\label{subsec:durer}

In 1514, Albrecht D\"urer included a $4\times4$ magic square (illustrated on the left-hand side of Figure~\ref{fig:durer}) in his engraving \emph{Melencolia I}---which is widely credited for bringing magic squares to prominence in Europe for the first time. Every row, column, and diagonal of D\"urer's square sums to the magic constant $34$; moreover, many (but not all) of the $2\times 2$ subsquares and pandiagonals of the square also sum to 34. While the square is not complete, it is associative. D\"urer's Magic Square carries the additional signature that the two middle entries of the bottom row, $15$ and $14$, spell out the year 1514 of the engraving. 

It is said that D\"urer did not make a most-perfect magic square because he wanted $15$~and~$14$ to be adjacent in the center of the bottom row, to represent the year $1514$ of the artwork; but this is not possible for a most-perfect $4\times4$ magic square over~$\mathbb Z$ (see Navas~\cite{Navas2020}). Indeed, by Theorem~\ref{thm:384orbit}: since the $384$ most-perfect squares form a single orbit under $\Aut(Q_4)$, and adjacency of specific entries is preserved by every element of this group, either \emph{all} $384$ most-perfect squares have $15$ and $14$ adjacent or \emph{none} do; and one checks directly that in the Khajuraho Magic Square, $15$ and $14$ are not adjacent---even on the folded~tesseract!

\begin{figure}[b]
\centering
\begin{subfigure}[b]{0.47\textwidth}
\centering
\begin{tikzpicture}[scale=0.7225]
\def\cs{1.25}
\draw[line width=0.9pt,xstep=\cs,ystep=\cs] (0,0) grid (4*\cs,4*\cs);
\node at (0.5*\cs,3.5*\cs) {\Large 16};
\node at (1.5*\cs,3.5*\cs) {\Large 3};
\node at (2.5*\cs,3.5*\cs) {\Large 2};
\node at (3.5*\cs,3.5*\cs) {\Large 13};
\node at (0.5*\cs,2.5*\cs) {\Large 5};
\node at (1.5*\cs,2.5*\cs) {\Large 10};
\node at (2.5*\cs,2.5*\cs) {\Large 11};
\node at (3.5*\cs,2.5*\cs) {\Large 8};
\node at (0.5*\cs,1.5*\cs) {\Large 9};
\node at (1.5*\cs,1.5*\cs) {\Large 6};
\node at (2.5*\cs,1.5*\cs) {\Large 7};
\node at (3.5*\cs,1.5*\cs) {\Large 12};
\node at (0.5*\cs,0.5*\cs) {\Large 4};
\node at (1.5*\cs,0.5*\cs) {\Large 15};
\node at (2.5*\cs,0.5*\cs) {\Large 14};
\node at (3.5*\cs,0.5*\cs) {\Large 1};
\end{tikzpicture}
\caption*{D\"urer's Magic Square (1514)}
\end{subfigure}
\hfill
\begin{subfigure}[b]{0.47\textwidth}
\centering
\begin{tikzpicture}[scale=0.7225]
\def\cs{1.25}
\draw[line width=0.9pt,xstep=\cs,ystep=\cs] (0,0) grid (4*\cs,4*\cs);
\node at (0.5*\cs,3.5*\cs) {\Large 7};
\node at (1.5*\cs,3.5*\cs) {\Large 13};
\node at (2.5*\cs,3.5*\cs) {\Large 9};
\node at (3.5*\cs,3.5*\cs) {\Large 5};
\node at (0.5*\cs,2.5*\cs) {\Large 3};
\node at (1.5*\cs,2.5*\cs) {\Large 11};
\node at (2.5*\cs,2.5*\cs) {\Large 1};
\node at (3.5*\cs,2.5*\cs) {\Large 2};
\node at (0.5*\cs,1.5*\cs) {\Large 8};
\node at (1.5*\cs,1.5*\cs) {\Large 12};
\node at (2.5*\cs,1.5*\cs) {\Large 10};
\node at (3.5*\cs,1.5*\cs) {\Large 4};
\node at (0.5*\cs,0.5*\cs) {\Large 16};
\node at (1.5*\cs,0.5*\cs) {\Large 15};
\node at (2.5*\cs,0.5*\cs) {\Large 14};
\node at (3.5*\cs,0.5*\cs) {\Large 6};
\end{tikzpicture}
\caption*{A Most-Perfect 15-14 Magic Square over $\mathbb F_{17}$}
\end{subfigure}
\caption{D\"urer's Magic Square (left), the first magic square to appear in European art, alongside a tribute to D\"urer (right): a {most-perfect} magic square over $\mathbb F_{17}$ sharing D\"urer's desired placement of $15$ and $14$ in the center of the bottom row---a placement that is  impossible for any most-perfect magic square over $\mathbb Z$.}
\label{fig:durer}
\end{figure}

However, we notice that, in our most-perfect magic square over $\mathbb F_{17}$ in Figure~\ref{fig:durer}, \pagebreak the numbers $15$ and $14$ \emph{are} adjacent. This leads to the most-perfect $4\times4$ magic square over $\mathbb F_{17}$ on the right-hand side of Figure~\ref{fig:durer}. Like D\"urer's square, this magic square in characteristic~$17$ also has many remarkable properties that D\"urer would have liked:

\begin{enumerate}
\item[{1)}] $15$ and $14$ sit adjacent in the center of the bottom row, exactly as D\"urer intended.
\item[{2)}] It is {\bf most-perfect} over $\mathbb F_{17}$: every row, column, and $2\times2$ subsquare---including those that wrap around---sums to a multiple of $17$ (though not always the same multiple when viewed as a plain integer: e.g., the row sums are $34$, $17$, $34$, $51$ in turn), while every diagonal and pandiagonal sums to exactly $34$ and every pair of entries two apart on a pandiagonal sums to exactly $17$.
\item[{3)}] Pairing the first two down-and-right pandiagonals against the last two reproduces D\"urer's own remarkable coincidence exactly (first observed by Barbeau~\cite{Barbeau1997}): the two resulting sets of eight entries are \emph{identical} to D\"urer's diagonal and non-diagonal sets, and consequently agree not only in sum ($68$) but also in sum of squares ($748$) and sum of cubes ($9248$), matching the numbers in D\"urer's own square precisely.
\item[{4)}] Analogous to D\"urer's groups of first two and last two rows (resp., columns), the two rows that individually sum to $34$ (resp., the two columns that individually sum to $34$) agree in sum of squares.
\item[{5)}] The remaining two rows and two columns, which individually sum to $17$ and $51$ rather than $34$, combine crosswise to give a further coincidence: the $17$-row together with the $51$-column matches the $51$-row together with the $17$-column, in both sum and sum of squares. 
\item[{6)}] Even the sixteen $2\times2$ subsquares (including wraparounds) echo this structure: the eight that individually sum to $34$ pair up perfectly under a diagonal shift of two rows and two columns, matching in both sum and sum of squares in every case.
\item[{7)}] The remaining eight $2\times2$ subsquares---four summing to $17$ and four summing to $51$---admit a further crosswise coincidence in the same spirit as Item~5: two disjoint choices of two $17$-subsquares together with two $51$-subsquares each yield the same sum as well as sum of squares, $1568$.
\end{enumerate}

\subsection{Most-perfect magic hypercubes of order $k$ and dimension $m$}

Theorems~\ref{thm:final1} and~\ref{thm:final2} (and Corollary~\ref{thm:countmostperfect}) represent the first time that not just existence but an exact count of a type of magic hypercubes of interest has been given across dimensions. 

The first-known (and still essentially the only) exact count of a type of normal magic squares of interest across {\it orders} (rather than dimensions) is also for ``most-perfect'' magic squares, due to the celebrated work of Ollerenshaw and Brée~\cite{OllerenshawBree1998}. 

{A {\it most-perfect magic hypercube} of order $k$ (where $k$ is even) and dimension $n\geq 2$ is an arrangement of the numbers $1,2,\ldots,k^n$ in a $k\times \cdots \times k$ array such that every orthogonal line has the same sum (i.e., it is magic), every $2\times 2$ orthogonal square has the same sum (i.e., it is compact), and every pair of entries where each coordinate differs by $k/2$ modulo $k$ has the same sum (i.e., it is complete).  

\pagebreak

Since the minimum value of $k$ for which a most-perfect magic hypercube of order $k$ can exist in any given dimension $m$ is $4$, the case of order $k=4$ has traditionally played a prominent role in the theory, and is the case treated here. Similarly, the minimum dimension $m$ of interest is $m=2$, which was the case treated by Ollerenshaw and Brée.} 

This suggests a potential common generalization of their work and ours, including a joining of a number of distinct techniques in both works---a natural and interesting potential direction that we hope may be considered in future work.

\subsection*{Acknowledgments}

I am very grateful to Professors G.~Bhowmik, P.~Deligne, H.~W.~Lenstra, A.~Navas, K.~Ramasubramanian, P.~Sarnak, and M.~D.~Srinivas for many helpful conversations. It is a pleasure to thank the Lodha Mathematical Sciences Institute for its hospitality during the summer of~2025, where much of this work was carried out, and the National Science Foundation (Grant  DMS-1001828) for its kind support.

\end{document}